\documentclass[12pt]{article}
\usepackage[T2A]{fontenc}
\usepackage[cp1251]{inputenc}
\usepackage[english]{babel}
\usepackage{amsmath,amssymb,comment}
\usepackage{graphicx}
\usepackage{amsthm}

\usepackage{indentfirst,latexsym,bm}
\graphicspath{{figures/}}

\def\R{{\mathbb R}}
\def\Z{{\mathbb Z}}
\def\N{{\mathbb N}}
\def\ds{\displaystyle}

\newtheorem{theorem}{Theorem}
\newtheorem{lemma}{Lemma}
\newtheorem{proposition}{Proposition}
\newtheorem{remark}{Remark}
\newtheorem{definition}{Definition}[section]

\newcommand{\const}{\operatorname{const}\nolimits}
\newcommand{\MAX}{\operatorname{MAX}\nolimits}
\newcommand{\Exp}{\operatorname{Exp}\nolimits}

\newcommand{\sn}{\operatorname{sn}\nolimits}
\newcommand{\ssc}{\operatorname{sc}\nolimits}
\newcommand{\cn}{\operatorname{cn}\nolimits}
\newcommand{\nc}{\operatorname{nc}\nolimits}
\newcommand{\sgn}{\operatorname{sgn}\nolimits}
\newcommand{\dc}{\operatorname{dc}\nolimits}
\newcommand{\dn}{\operatorname{dn}\nolimits}
\newcommand{\E}{\operatorname{E}\nolimits}

\newcommand{\K}{\operatorname{K}\nolimits}
\def\tcut{t_{\operatorname{cut}}}
\def\tsupr{t_{\operatorname{supr}}}
\def\tmax{t_{\MAX}}

\begin{document}

\title{\large Extremals in the Engel group with a sub-Lorentzian metric
\thanks{Supported by NSFC (No.11071119, No.11401531); NSFC-RFBR (No. 11311120055).}}
\author{\normalsize A. A. Ardentov$^1$, Tiren Huang$^2$, Yu. L. Sachkov$^1$,
Xiaoping Yang$^3$\\
\scriptsize 1 Program Systems Institute of the Russian Academy of Sciences, Pereslavl'-Zalesski\v{\i}, Russia\\
\scriptsize 2 Department of Mathematics, Zhejiang Sci-Tech
University, Hangzhou 310018, China\\
\scriptsize 3 Department of Applied Mathematics, Nanjing University
of Science $\&$ Technology, Nanjing 210094,China\\
\scriptsize E-mail: aaa@pereslavl.ru, htiren@ustc.edu.cn,
sachkov@sys.botik.ru, yangxp@njust.edu.cn.}
\date{}
\maketitle
 \fontsize{12}{22}\selectfont\small
\paragraph{Abstract:}
Let $E$ be the Engel  group  and $D$ be a rank 2 bracket generating left
invariant distribution with a Lorentzian metric, which is a
nondegenerate  metric of index 1. In this paper,  we first prove that timelike normal extremals are locally maximizing. Second, we obtain a parametrization of timelike, spacelike, lightlike normal extremal trajectories by Jacobi functions. Third, a discrete symmetry group and its fixed points which are Maxwell points of of timelike and spacelike normal extremals, are described. An estimate for the cut
time (the time of  loss of optimality) on extremal trajectories is derived on this basis.
\\[10pt]
\\[10pt]

\section{\normalsize Introduction}
A sub-Riemannian structure on a manifold $M$ is given by a smoothly
varying distribution $D$ on $M$ and a smoothly varying positive
definite metric $g$ on the distribution. The triple $(M,D,g)$ is
called a \emph{sub-Riemannian manifold}, which has been applied in
control theory, quantum physics, C-R geometry and  other areas.
Some efforts have been made to generalize sub-Riemannian manifold.
One of them leads to the following question: what kind of
geometrical features the mentioned triple will have if we change the
positive definite metric to an indefinite nondegenerate metric?
It is natural to start with the Lorentzian metric of index 1. In
this case the triple: manifold, distribution and Lorentzian metric
on the distribution is called a \emph{sub-Lorentzian manifold} by
 analogy to a Lorentzian manifold. For details concerning the
\emph{sub-Lorentzian geometry}, the reader is referred to
\cite{M.Golubitsky3}.  To our knowledge, there are only few works
devoted to this subject (see \cite{Chang,M.Golubitsky3,M.Golubitsky4,
M.Golubitsky5,M.Golubitsky6,Korolko.A}). In \cite{Chang}, Chang,
Markina, and Vasiliev  systematically studied   geodesics in
an anti-de Sitter space with a sub-Lorentzian metric and a
sub-Riemannian metric respectively. In \cite{M.Golubitsky5},
Grochowski computed reachable sets starting from a point in the
Heisenberg sub-Lorentzian manifold on $\mathbb{R}^3$. It was shown
in \cite{Korolko.A} that the Heisenberg group $\mathbb{H}$ with a
Lorentzian metric on $\mathbb{R}^3$ possesses the uniqueness of
Hamiltonian geodesics of time-like or space-like type.

The Engel group was first named by Cartan \cite{Cartan} in 1901. It
is a prolongation of a three dimensional contact manifold, and is a
Goursat manifold. In \cite{Sachkov,Sachkov1,Sachkov2}  A. Ardentov and Y. Sachkov
computed minimizers on the sub-Riemannian Engel group. In
the present article, we study the Engel group furnished with a
sub-Lorentzian metric. This is an interesting example of
sub-Lorentzian manifolds, because the Engel group is the simplest
sub-Lorentzian manifold with nontrivial abnormal extremal
trajectories, and the vector distribution of the Engel group is not
$2-$generating, its growth vector is $(2,3,4)$.
 We first prove that    timelike normal extremals are locally maxiziming. Second, we  study the  normal extremals. By using the Pontryagin Maximum Principle, the subsystem
for costate variables of normal Hamiltonian system is reduced
to equations  similar to pendulum. Expressions of  timelike, spacelike and lightlike normal extremal trajectories are obtained in terms fo Jacobi functions. Third, a discrete symmetry group and its fixed points, which are Maxwell points, are described. An estimate for the cut
time (the time of  loss of optimality) on extremal trajectories is derived
on this basis.

The structure of this paper is as follows.
Section 2 contains some preliminaries as well as definitions of
sub-Lorentzian manifolds and the Engel group. In Section 3, we  prove that timelike normal extremals are locally maxiziming. In Section 4 we get a description of lightlike normal extremal trajectories. In Section 5, we find a parametrization of timelike normal extremal trajectories and describe the exponential map providing the parametrization for all the timelike normal extremal trajectories, then we describe symmetries of the exponential map and investigate the corresponding Maxwell points, which are fixed points of these symmetries. In Section 6,  a parametrization of spacelike normal extremal trajectories are obtained, and the exponential map providing parametrization for all the spacelike normal extremal trajectories is described, then we describe symmetries of the exponential map and investigate the corresponding Maxwell points, which are fixed points of these symmetries. On the basis of this study, we prove upper bounds on the cut time.

\section{Preliminaries}\label{sec:bd}
A sub-Lorentzian manifold is a triple $(M,D,g)$, where $M$ is a
smooth $n-$dimensional manifold, $D$ is a smooth distribution on $M$
and $g$ is a smoothly varying Lorentzian metric on $D$. For each point
$p\in M$, a vector $v\in D_p$ is said to be horizontal. An
absolutely continuous curve $\gamma(t)$ is said to be horizontal if
its derivative $\gamma'(t)$ exists almost everywhere and lies in
$D_{\gamma(t)}$.

A vector $v\in D_p$ is said to be time-like if $g(v,v)<0$;
space-like if $g(v,v)>0$ or $v=0$; null (lightlike) if $g(v,v)=0$ and
$v\neq 0$; and non-space-like if $g(v,v)\leq 0$. A curve is said to
be time-like if its tangent vector is time-like a.e.; similarly,
space-like, null, and non-space-like curves can be defined.

By a time orientation of $(M,D,g)$, we mean a continuous time-like
vector field on $M$. From now on, we assume that $(M,D,g)$ is
time-oriented. If $X$ is a time orientation on $(M,D,g)$, then a
non-space-like $v\in D_p$ is said to be future directed if
$g(v,X(p))<0$, and past directed if $g(v,X(p))>0$. Throughout this
paper, ``f.d." stands for ``future directed", ``t." for
``time-like", and ``nspc." for ``non-space-like".

Let $v,w\in D$ be two nspc. vectors, we have a following
reverse Schwarz-inequality \cite{O'Neill}:
\begin{equation}
|g(v,w)|\geq ||v||\cdot||w||,
\end{equation}
where $||v||=\sqrt{|g(v,v)|}$, the equality holds if and only if $v$ and $w$
are linearly dependent. If $v$ is a nonspacelike f.d. vector and $w$
is a nonspacelike p.d. vector, it is  easy  to see that $g(v,w)\geq
0$.

We introduce the space $H_{\gamma(t)}$ of horizontal nspc.\ curves:
$$H_{\gamma(t)}=\{\gamma:[0,1]\rightarrow M|g(\gamma'(t),\gamma'(t))\leq0,
\gamma'(t)\in D_{\gamma(t)} \hbox{ for almost all }t \in
[0,1]\}.$$ The sub-Lorentzian length of a horizontal nspc.\ curve
$\gamma(t)$ is defined as follows:
$$l(\gamma)=\int_0^1 \|\gamma'(t)\|dt,$$
where $\|\gamma'(t)\|=\sqrt{|g(\gamma'(t),\gamma'(t))|}$ is computed
using the Lorentzian metric on the horizontal spaces
$D_{\gamma(t)}$. We use the length to define the sub-Lorentzian
distance $d_U(q_1,q_2)$ with respect to a set $U\subset M$ between
two points $q_1,q_2,\in U$:

\[ d_U(q_1,q_2)=\bigg\{\begin{array}{ll} \sup\{l(\gamma), \gamma\in
H_U(q_1,q_2)\} & \hbox{if} \  \ H_U(q_1,q_2)\neq \emptyset \\ 0
&\hbox{if}\  \ H_U(q_1,q_2)=\emptyset,
\end{array}
\]
where $H_U(q_1,q_2)$ is the set of all nspc.f.d\ curves contained in
$U$ and joining $q_1$ and $q_2$.

A distribution $D\subset TM$ is called bracket generating if any
local frame $\{X_i\}_{1\leq i\leq r}$ for $D$, together with all of
its iterated Lie brackets $[X_i,X_j],[X_i,[X_j,X_k]],\cdots$ span
the tangent bundle $TM$. Bracket generating distributions are
 also called completely nonholonomic or distributions
satisfying H$\ddot{o}$rmander's condition.

\begin{theorem}(Rashevsky-Chow) Fix a point $q\in M$. If a distribution $D\subset TM$ is bracket
generating then the set of points that can be connected to $q$ by a
horizontal curve is the component of $M$ containing $q$.
\end{theorem}

By the  Rashevsky-Chow Theorem, we know that if $D$ is bracket generating and
$M$ is connected, then any two points of $M$ can be joined by a
horizontal curve.

In the Lorentzian geometry   the lengh minimizer problem is trivial because arbitrary two points can be connected by a piecewise
lightlike curve of zero length. However, there do exist timelike curves with
maximal length which are timelike geodesics \cite{O'Neill}. Thus it is natural to study sub-Lorentzian length maximizers.

A nspc.\ curve is said to be a maximizer if it realizes the distance
between its endpoints. We also use the name $U$-extremal for a curve
in $U$ whose each suitably short sub-arc is a $U$-maximizer. A nspc.\ curve will be called geodesic if  it is a local maximizer.

The Hamilton function $H: T^*M\rightarrow M$ be defined as
\begin{equation}\label{Hamilton function}
H(x,\lambda)=-\frac12<\lambda, X_1(x)>^2+\frac12\sum_{i=2}^r<\lambda, X_i(x)>^2,\
\ (x,\lambda)\in T^*M.
\end{equation}

Let $U\subset M$ be an open neighborhood, $\varphi: U\subset
M\rightarrow R$ is a smooth function.
\begin{definition}
The horizontal gradient $\nabla_H\varphi$ of $\varphi$ is a smooth
horizontal vector field on $U$ such that for each $p\in U$ and $v\in
H$ we have $\partial_v\varphi(p)=g(\nabla_H\varphi(p),v).$
\end{definition}

Locally we can write
\begin{equation}\label{horizontal gradient}
\nabla_H\varphi=-(\partial_{X_1}\varphi)X_1+\sum_{i=2}^r(\partial_{X_i}\varphi)X_i
\end{equation}

\begin{definition}\label{normal extremals}
A normal extremal trajectory  in the sub-Lorentzian manifold $(M,D,g)$ is a curve $\gamma(t): [a,b]\rightarrow M$
that admits a lift $\Gamma : [a,b] \rightarrow T^*M $ which is a solution of the Hamiltonian system with the sub-Lorentzian Hamiltonian $H(x, \lambda)$. In this case, we say that  $\Gamma(t)$  is a normal lift of $\gamma(t)$.
\end{definition}

Now we introduce the Engel group $E$ with coordinates $q=(x_1,x_2,y,z)\in \mathbb{R}^4$. The group
law is denoted by $\odot$ and is defined as follows:
\begin{align}
&(x_1,x_2,y,z)\odot(x_1',x_2',y',z')\nonumber\\
&=\left(x_1+x_1',x_2+x_2',y+y'+\frac{x_1x_2'-x_1'x_2}{2},z+z'+\frac{x_2x_2'}{2}(x_2+x_2')+x_1y'+\frac{x_1x_2'}{2}(x_1+x_1')\right).\nonumber
\end{align}

A vector field $X$ on $E$ is said to be left-invariant if it satisfies
$dL_qX(e)=X(q)$, where $L_q$ denotes the left translation
$p\rightarrow L_q(p)=q\odot p$ and $e$ is the identity of $E$. This
definition implies that any left-invariant vector field on $E$ is a
linear combination with constant coefficients of the following vector fields:
\begin{align}\label{frame}
&X_1=\frac{\partial}{\partial x_1}-\frac {x_2}{2}
\frac{\partial}{\partial y};\ \  X_2=\frac{\partial}{\partial
x_2}+\frac {x_1} {2}\frac{\partial}{\partial
y}+\frac{x_1^2+x_2^2}{2}\frac{\partial}{\partial z};\nonumber\\
&X_3=\frac{\partial}{\partial y}+x_1\frac{\partial}{\partial z};\ \
\ \  X_4=\frac{\partial}{\partial z}.
\end{align}
The distribution $D=span\{X_1,X_2\}$ on $E$ satisfies the bracket
generating condition, since $ X_3=[X_1,X_2], \ X_4=[X_1,X_3].$ The
Engel group is a nilpotent Lie group, since $
[X_1,X_4]=[X_2,X_3]=[X_2,X_4]=0.$ We define a smooth Lorentzian
metric on $D$ by
\begin{align}\label{metric}
g(X_1,X_1)=-1,\ \ \ \  g(X_2,X_2)=1,\ \ \ \  g(X_1,X_2)=0.
\end{align}

\section{ Local maximizing property of normal extremals}
In this secton, we prove that timelike normal extremals are locally maximizing.

\begin{theorem}
Let $\gamma(t)$ be a t.f.d. normal extremal defined on an interval
$[a, b]$. Then for each $t\in[a,b]$, there exists a neighborhood $U$
of $\gamma(t)$ such that the restriction of $\gamma$ to $U$ is a
unique maximizer connecting its endpoints.
\end{theorem}
\begin{proof}
Let $\Gamma(t)=(\gamma(t),\lambda(t))\in T^*M$ be a normal lift of $\gamma(t)$,
then we can assume that $H(\gamma,\lambda)=-\frac12, \forall t\in(a,b)$. To show that
$\gamma$ is locally maximizing, we need only to prove that for every
$c\in (a,b)$, there exist $J=(t_1, t_2)\subset (a,b)$  a small interval of
$c$, such that the restriction of $\gamma(t)$ to $J$ is maximizer. Let
$p_c=\gamma(c)$ and $\lambda_c=\lambda(c)$. Pick a smooth
hypersurface $S\subset M$ of dimension $n-1$ in $M$ passing through
$p_c$ such that $\lambda_c$ vanishes on $T_{p_c}S$. Let $\bar
\lambda$ be a smooth one-form on an open  neighborhood $\Omega$ of
$p_c$ such that $\bar\lambda(p_c)=\lambda$ and $\bar \lambda (p)$
annihilates $T_pS$ and $H(p,\bar\lambda)=-\frac12$ for all $p\in
S\cap\Omega$.

Let $\Gamma_p=(\gamma_p,\lambda_p)$ be the solution of
$\dot\Gamma(t)=\overrightarrow{H}(\Gamma(t))$,
$\Gamma(c)=(p,\bar\lambda), \  p\in S\cap \Omega$. It is clear that
$\Gamma=\Gamma_{p_c}$. Since $\dot\gamma\notin T_{p_c}S$, by the
Implicit Function Theorem, we can get, if $\varepsilon > 0$ is
small enough and $W$ is a sufficiently small open neighborhood of
$p_c$ in $S$, that there exist a  map
\begin{align}
\nu: &(c-\varepsilon,c+\varepsilon)\times W\rightarrow M,\nonumber\\
     & (t,p)\rightarrow  \gamma_p(t),\nonumber
\end{align}
that maps $(c-\varepsilon, c+ \varepsilon)\times\ W$ diffeomorphically
onto an open neighborhood $U$ of $p_c$ in $M$. Define a smooth
function $\varphi: U\rightarrow R$ and a 1-form $\theta$ on $U$ by
letting $\varphi(x)=t$ and $\theta(x)=\lambda_p(t) $ if
$x=\gamma_p(t)$.

We will prove that $d\varphi = -\theta$. Let $X$ be a vector field
on $U$ such that
\begin{align}
X(\gamma(t))=d\pi_*(\Gamma_p(t))\overrightarrow{H}(\Gamma_p(t)),\nonumber
\end{align}
then we know that $\dot \gamma_p(t)=X(\gamma_p(t))$. Let
$\overrightarrow{H}_X$ be the Hamiltonian lift of $X$, we can show
that
\begin{align}
\overrightarrow{H}_X(\Gamma_p(t))=\overrightarrow{H}(\Gamma_p(t)),\quad
\forall p\in W, \quad t\in(c-\varepsilon,c+\varepsilon).\nonumber
\end{align}
 Indeed, let $a_i=(\theta(x),X_i(x))$ for $i = 1,..., k$, if
$x\in U$, then
\begin{equation}
X=-a_1X_1+\sum_{i=2}^ka_iX_i.\nonumber
\end{equation}
Since $H$ is constant along integral curves of $\overrightarrow{H}$, then
\begin{equation}
H(\Gamma_p(t))=-\frac12,\ \ \  \forall
 (t,p)\in(c-\varepsilon,c+\varepsilon)\times W,\nonumber
 \end{equation}
so we have that
\begin{equation}
-a_1^2(x)+\sum_{i=2}^ka_i^2(x) =-1,\ \
-a_1(x)\overrightarrow{a_1}(x)+\sum_{i=2}^ka_i(x)\overrightarrow{a_i}(x)=0,
\ \ \forall x\in U.
\end{equation}

Since
\begin{align}
\overrightarrow{H}(x,\lambda)=-\langle\lambda,
X_1\rangle\overrightarrow{H_{X_1}}(x,\lambda)+\sum_{i=2}^k\langle\lambda,
X_i\rangle\overrightarrow{H_{X_i}}(x,\lambda),
\end{align}
in particular
$$\overrightarrow{H}(\gamma_p(t),\lambda_p(t))=-a_1(\gamma_p(t))\overrightarrow{H}_{X_1}(\gamma_p(t),
\lambda_p(t))+\sum_{i=2}^ka_i(\gamma_p(t))\overrightarrow{H}_{X_i}(\gamma_p(t),\lambda_p(t)).$$

On the other hand,
\begin{align}
\overrightarrow{H}_X(x,\lambda)=-a_1(x)\overrightarrow{H}_{X_1}(x,\lambda)+\sum_{i=2}^ka_i(x)\overrightarrow{H}_{X_i}(x,\lambda)-
\langle\lambda,X_1\rangle\overrightarrow{a_1}(x)+\langle\lambda,X_i\rangle\overrightarrow{a_i}(x).\nonumber
\end{align}
Therefore
\begin{align}
\overrightarrow{H_X}(\gamma_p,\lambda_p)& = -a_1(\gamma_p)\overrightarrow{H}_{X_1}(\gamma_p,\lambda_p)+
\sum_{i=2}^ka_i(\gamma_p)\overrightarrow{H}_{X_i}(\gamma_p,\lambda_p)\nonumber
\\&\qquad-
\langle\lambda_p,X_1(\gamma_p)\rangle\overrightarrow{a_1}(\gamma_p)+\sum_{i=2}^k\langle\lambda_p,X_i(\gamma_p)\rangle\overrightarrow{a_i}(\gamma_p)\nonumber\\
&=a_1(\gamma_p)\overrightarrow{H}_{X_1}(\gamma_p,\lambda_p)+
\sum_{i=2}^ka_i(\gamma_p)\overrightarrow{H}_{X_i}(\gamma_p,\lambda_p)\nonumber\\
&\qquad-
a_1(\gamma_p)\overrightarrow{a_1}(\gamma_p)+\sum_{i=2}^ka_i(\gamma_p)\overrightarrow{a_i}(\gamma_p)\nonumber\\
&= a_1(\gamma_p)\overrightarrow{H}_{X_1}(\gamma_p,\lambda_p)+
\sum_{i=2}^ka_i(\gamma_p)\overrightarrow{H}_{X_i}(\gamma_p,\lambda_p)\nonumber\\
&= \overrightarrow{H}(\gamma_p,\lambda_p).
\end{align}
So, for $p\in W$, $\Gamma_p$ is also the solution of
$$\dot\Gamma(t)=\overrightarrow{H}_X(\Gamma(t)),\ \  \Gamma(c)=(p,\bar\lambda(p)).$$

Let $\phi_t$ be the flow of $X$ on $U$ for
$t\in(c-\varepsilon,c+\varepsilon)$ and let $\phi_c(p)=p$ for $p\in U$.
For $p\in W$, we have $\gamma_p(t)=\phi_t(p)$.

For each $t\in(c-\varepsilon,c+\varepsilon)$, let $W_t=\{x\in U:
\varphi(x)=t\}$. Then $W_t$ is a smooth hypersurface, and
$W_t=\phi_t(W)$ and $T_xW_t=d\phi_t(p)T_pW$ for $x=\gamma_p(t)$.

We now show that $ d\varphi=-\theta$ on $U$. Let $x\in U$ and $p\in
W, t\in(c-\varepsilon,c+\varepsilon)$, be such that $x=\gamma_p(t)$,
then $x\in W_t$. Let $w\in T_xM$ and $v\in T_pM$  be such than
$w=d\phi_t(p)v$. Define $h(s)=d\phi_s(p)v$, then $h(s)\in
T_{\gamma_p(s)}W_s$ and $h(t)=w$. It is easy to see that $h(s)$
satisfies the following equation:
$$\dot h(s)=\frac{\partial X}{\partial x}(\gamma_p(s))h.$$
Since $\lambda_p$ satisfies the Hamiltonian equations, so
$$\dot \lambda_p(s)=-\frac{\partial H_X}{\partial x}(\gamma_p(s))=-\lambda_p\frac{\partial X}{\partial x}(\gamma_p(s)),$$
hence the fucntion $s\mapsto \langle\lambda_p(s),h(s)\rangle$ is constant on
$(c-\varepsilon,c+\varepsilon)$.

If $w\in T_xW_t$, then $v\in T_pW$. Since $\varphi(x)$ is constant
on $W_t$, then $\langle d\varphi(x),w\rangle=0$. Since $\langle\lambda_p(c),v\rangle=0$, we have
that
\begin{equation}\label{1}
\langle\theta(x),w\rangle=\langle\lambda_p(t),h(t)\rangle=\langle\lambda_p(c),h(c)\rangle=\langle\lambda_p(c),v\rangle=0=-\langle d\varphi(x),w\rangle.\end{equation}
If $w=\dot\gamma_p(t)$, then $v=\dot\gamma_p(c)$. Since
$\langle d\varphi(x),\dot\gamma_p(t)\rangle=1$, we have that

\begin{equation}\label{2}
\langle\theta(x),w\rangle=\langle\lambda_p(t),h(t)\rangle=\langle\lambda_p(c),h(c)\rangle=-1=-\langle d\varphi(x),w\rangle.
\end{equation}
Since $T_xM=T_xW_t\oplus\{\dot\gamma_p(t)\}$, we conclude that
$\theta(x)=-d\varphi(x)$.

By (\ref{horizontal gradient}), (\ref{1}) and (\ref{2}), we can get
that
\begin{equation}
\nabla_H\varphi(t)=-\dot\gamma_p(t).
\end{equation}
Since $\gamma(t)$ is a t.f.d curve, $\nabla_H\varphi(t)$ is a p.d.
vector field on $U$.

Since $H(x, \theta(x))=-\frac12$ and $||\theta(x)||=-1$, we have
shown that $\varphi$ satisfies the Hamiltonian equation $H(x,
d\varphi(x))=-\frac12$ and $||d\varphi(x)||=-1$.

Next, choose $t_1$ and $t_2$ in the interval $(c-\varepsilon,c+\varepsilon)$.
Now, let $\eta: [0,\alpha]$ be a t.f.d. curve in $U$ with
$\eta(0)=\gamma(t_1)$ and $\eta(\alpha)=\gamma(t_2)$, then we  have
\begin{align}
L(\gamma|_{[t_1,t_2]})&=t_2-t_1=\varphi(\gamma(t_2))-\varphi(\gamma(t_1))\nonumber\\
&=\varphi(\eta(0))-\varphi(\eta(\alpha))=\int_0^\alpha(d\varphi(\eta(s)))ds=\int_0^\alpha g(\dot\eta,\nabla_h\varphi)\nonumber\\
&\geq \int_0^\alpha||\dot\eta(s)||ds=L(\eta|_{[0,\alpha]}).
\end{align}
By the reverse Schwarz inequality, $L(\gamma)=L(\eta)$ holds if and only if $\eta$
can be reparameterized as a trajectory of $-\nabla_H\varphi$. This
completes the proof.
\end{proof}

\begin{remark}
Let $\gamma(t)$ be a t.p.d. normal extremal. We also can prove that $\gamma(t)$ is a the unique
 local maximizer by the above method.
\end{remark}

\section{Sub-Lorentzian extremal trajectories}
In this section, we calculate sub-Lorentzian extremal trajectories  by applying the Pontryagin  maximum  principle.
We can assume that the
initial point is the origin by invariance under left translations of the Engel group, i.e., $x_1(0)=x_2(0)=y(0)=z(0)=0.$
 Let us introduce the
vector of costate variables $\xi=(\xi_0,\xi_1,\xi_2,\xi_3,\xi_4)$
and define the Hamiltonian function
\begin{align}
H(\xi,q(t),u)=\xi_0\frac{-u^2_1+u^2_2}{2}+\xi_1u_1+\xi_2u_2+\xi_3\frac{x_1u_2-x_2u_1}{2}+\xi_4\frac{x_1^2+x_2^2}{2}u_2.
\end{align}
From the Pontryagin maximum principle for this Hamiltonian function we
obtain a Hamiltonian system for the costate variables:
\begin{align}
\dot \xi_1=-H_{x_1}=-\frac{\xi_3u_2}{2}-\xi_4x_1u_2,\quad
\dot\xi=-H_{x_2}=\frac{\xi_3u_1}{2}-\xi_4x_2u_2,\quad
\dot\xi_3=\dot\xi_4=0,
\end{align}
the maximality condition
\begin{align}\label{max pri}
\max_{u\in\mathbb{R}^2}H(\xi(t),{q}(t),{u}(t))=H(\xi(t),\hat{q}(t),\hat{u}(t)),\
\ \xi_0\leq 0,
\end{align}
where $\hat{u}(t),\hat{q}(t)$ is the optimal process, and the
condition $\xi(t)\neq 0$ of  nontriviality of the costate
variables.

For sub-Lorentzian extremal trajectories, there are abnormal extremal trajectories (i.e. $\xi_0=0$) and normal extremal trajectories (i.e. $\xi_0=-1$).
Abnormal extremal trajectories were described in \cite{Sachkov,Cai}

Now we look at the normal case $\xi_0=-1$.  It follows from the
 maximality condition (\ref{max pri}) that
$H_{u_1}=H_{u_2}=0$. Hence
\begin{align}
u_1=-(\xi_1-\frac{x_2\xi_3}{2}),\quad
u_2=\xi_2+\frac{\xi_3x_1}{2}+\frac{\xi_4(x_1^2+x_2^2)}{2}.
\end{align}

Let $h_i=\langle \xi,X_i\rangle,i=1, \cdots,4,$ be the Hamiltonians corresponding to
the basis vector fields $X_1,X_2,X_3,X_4$ in the tangent space $T_qE$ and
linear on the fibres of the cotangent space $T_q^*E$:
\begin{align}
h_1=\xi_1-\frac{x_2}{2}\xi_3,\quad
h_2=\xi_2+\frac{x_1}{2}\xi_3+\frac{x_1^2+x_2^2}{2}\xi_4,\quad
h_3={\xi_3}+x_1\xi_4,\quad
h_4=\xi_4.
\end{align}
So $u_1=-h_1$ and $u_2=h_2$.

According to the Pontryagin maximum principle, we get the following
Hamilton equations in the new coordinates:

\begin{align}
&\dot x_1 =\frac{\partial H}{\partial \xi_1}=-(\xi_1-\frac{x_2}{2}\xi_3)=-h_1,\label{ham1}\\
&\dot x_2=\frac{\partial H}{\partial \xi_2}=\xi_2+\frac{x_1}{2}\xi_3+\frac{x_1^2+x_2^2}{2}\xi_4=h_2, \\
&\dot y=\frac{\partial H}{\partial
\xi_3}=h_1\frac{x_2}{2}+h_2\frac{x_1}{2}=\frac 1
2(x_1h_2+x_2h_1),\\
&\dot z=\frac{\partial H}{\partial \xi_4}=\frac{x_1^2+x_2^2}{2}h_2,\\
&\dot h_1=\dot\xi_1-\frac{\dot x_2}{2}\xi_3=-h_2h_3,\\
&\dot h_2=\dot \xi_2+\frac{\dot x_2}{2}\xi_3+\xi_4(x_1\dot x_1+x_2\dot x_2)=-h_1h_3,\\
&\dot h_3=\dot x_1\xi_4=-h_1h_4,\\
&\dot h_4=0.\label{ham2}
\end{align}

Associated with the expression of $H$, we conclude that a sub-Lorentzian extremal is timelike
if $H < 0$; spacelike if $H > 0$; lightlike if $H = 0$.

For the lightlike case, by the definition, we have $H=\frac{1}{2}(-h_{1}^{2}+h_{2}^{2})=0$,
thus $h_{2}=\pm h_{1}.$ If $h_{2}=h_{1}$, then lightlike trajectories
satisfy the ODE:
$$\dot{\gamma}=-h_{1}(X_{1}-X_{2}), $$
that is, they are reparameterizations of the one-parametric subgroup of
the field $X_{1}-X_{2}$. We assume $\dot{\gamma}=X_{1}-X_{2}$, so
\begin{align}
\dot{x}_{1}=1,\quad \dot{x}_{2}=-1,\quad \dot{y}=-\frac{1}{2}(x_{1}+x_{2}),\quad
\dot{z}=-\frac{1}{2}(x_{1}^{2}+x_{2}^{2}),\nonumber
\end{align}
thus
\begin{align}
x_{1}=t,\quad x_{2}=-t,\quad y=0,\quad z=-\frac{1}{3}t^{3}.\nonumber
\end{align}
If $h_{2}=-h_{1}$, similarly, we obtain
\begin{align}
x_{1}=t,\quad x_{2}=t,\quad y=0,\quad z=\frac{1}{3}t^{3}.\nonumber
\end{align}
In conclusion, we get the following theorem:
\begin{theorem}
Lightlike horizontal extremal trajectories starting from origin
 are reparameterizations of the curves:
\begin{align}
x_{1}=t,\quad x_{2}=\pm t,\quad y=0,\quad z=\pm\frac{1}{3}t^{3}.\nonumber
\end{align}
\end{theorem}


Timelike and spacelike normal extremal trajectories are considered in the following two sections to discuss them.

\section{Timelike normal extremal trajectories }
In the timelike case $H= \frac{1}{2}\big(-h_1^2+h_2^2\big)<0$, we consider extremals on the level surface $\{H=-1/2\}$ and introduce coordinates $(\theta, c, \alpha)$ on this surface as follows:
\begin{align*}
&h_1 = \pm \cosh \theta,\ \ \  \ h_2 = \sinh \theta, \qquad h_3=c,\ \ \  \ h_4=\alpha.
\end{align*}

Notice that we have the following symmetry of the Hamiltonian system~(\ref{ham1})--(\ref{ham2}):
\begin{align}
\varepsilon^0 \colon (h_1, h_2, h_3, h_4, x_1, x_2, y, z) \mapsto (-h_1, h_2, -h_3, h_4, -x_1, x_2, -y, z). \label{tleps0}
\end{align}
So we consider case $h_1 = \cosh \theta >0$ without loss of generality in the sequel.

In the variables $(\theta, c, \alpha, x_1, x_2, y, z)$ on the level surface $\{H=-1/2\}$ the Hamiltonian system~(\ref{ham1})--(\ref{ham2}) assumes the following form:
\begin{align}
&\dot{x}_1 = - \cosh \theta, \label{dotx1t}\\
&\dot{x}_2 = \sinh \theta, \label{dotx2t}\\
&\dot{y} = \frac{x_2  \cosh \theta + x_1 \sinh \theta}{2}, \label{dotyt}\\
&\dot{z} = \frac{x_1^2+x_2^2}{2}\sinh \theta, \label{dotzt}\\
&\dot{\theta} = -  c, \label{dottht}\\
&\dot{c} = - \alpha \cosh \theta, \label{dotct}\\
&\dot{\alpha} = 0. \label{dotat}
\end{align}

Note that the subsystem for the costate variables reduces to the equations
\begin{align}
&\ddot{\theta} = \alpha \cosh \theta, \qquad \dot{\alpha} = 0, \label{pend}
\end{align}
whose phase portrait for $\alpha = 1$ and $\alpha = -1$ is given in Fig.~\ref{pendulum+-}.

\begin{figure}[htbp]
\centering
\includegraphics[width=0.47\linewidth]{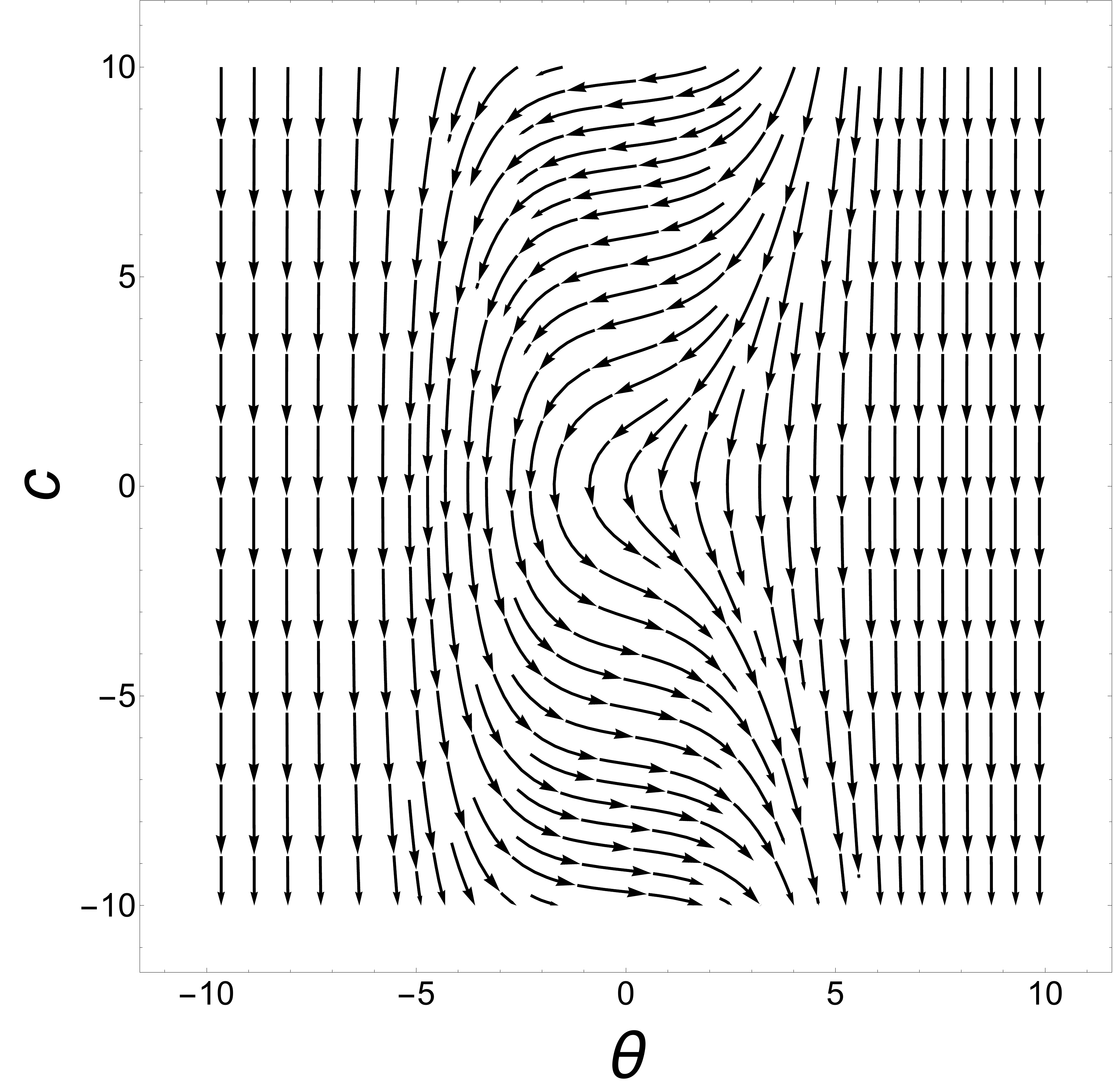}
\qquad
\includegraphics[width=0.47\linewidth]{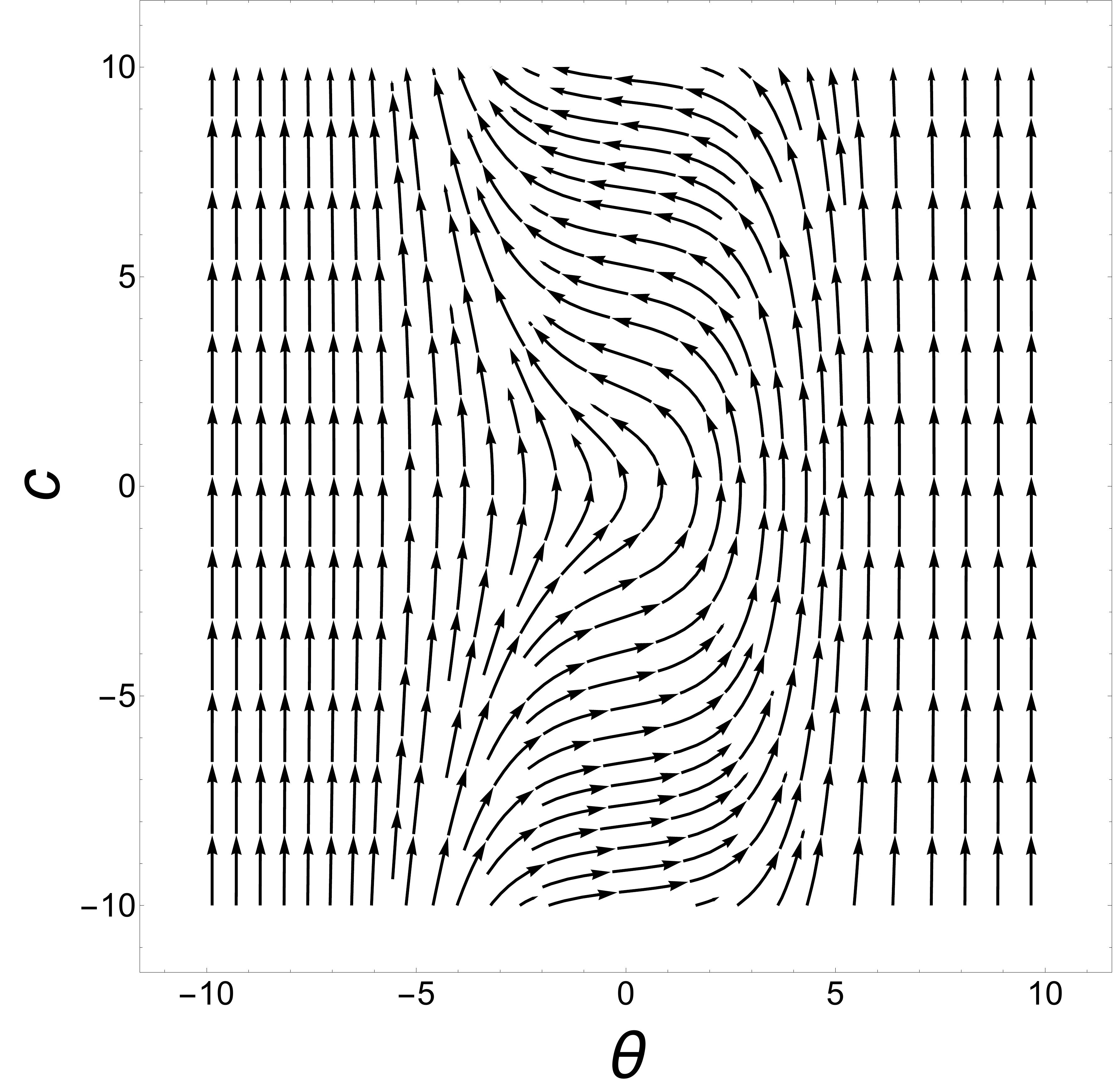}
\caption{
Phase portrait of vertical subsystem~(\ref{pend}) for $C^+$ and $C^-$ (timelike case)
}\label{pendulum+-}
\end{figure}
Let us introduce the energy integral:
\begin{align}
&E = \frac{h_3^2}{2} - h_2 h_4 = \frac{c^2}{2} - \alpha \sinh \theta, \qquad \dot{E} = h_3 \dot{h}_3 - h_4 \dot{h}_2 = 0.
\end{align}
The set of all normal extremal paths is parametrized by points of the set
\begin{align*}
C&=T_{q_0}^* M \cap \{H=-1/2, h_1 >0\}=\big\{(h_1, h_2, h_3, h_4) \in \R^4 ~ | ~ h_1^2-h_2^2=1, h_1 > 0\big\}\\
&=\big\{(\theta, c, \alpha) \in \R^3\big\}.
\end{align*}

The set $C$ has the following decomposition into invariant sets of equations~(\ref{dottht})--(\ref{dotat}):
\begin{align*}
C &= C_0^0 \cup C^0 \cup C^+ \cup C^-, \quad \lambda = (\theta, c, \alpha), \\
C_0^0&= \{\lambda \in C ~ | ~ c = 0, \alpha = 0\}, \\
C^0 &= \{\lambda \in C ~ | ~ c \neq 0, \alpha = 0\}, \\
C^+ &= \{\lambda \in C ~ | ~ \alpha > 0\}, \\
C^- &= \{\lambda \in C ~ | ~ \alpha < 0\}.
\end{align*}

Consider the general case $\alpha \neq 0$. Trajectories of vertical subsystem~(\ref{dottht})--(\ref{dotat}) for $C^+$ and $C^-$ are symmetric (see~Fig.~\ref{pendulum+-}). This symmetry (see~Section~\ref{symm}) has the following description:
\begin{align}
\varepsilon^1 \colon (\alpha, c, \theta, x_1, x_2, y, z) \mapsto (-\alpha, -c, -\theta, x_1, -x_2, -y, -z). \label{tleps1}
\end{align}

Thus, it is enough to integrate the Hamiltonian system in the case $C^+$ due to the symmetry $\varepsilon^1$. We introduce new coordinates  $(\varphi, E, \alpha)$ in the subset $C^+$ which rectify the vertical subsystem~(\ref{dottht})--(\ref{dotat}):
\begin{align*}
&\ae=\sqrt{\frac{\sqrt{E^2 + \alpha^2}}{2}}, \qquad &&k^2 = \frac{1}{2} + \frac{E}{4 \ae^2} \in (0,1), \\
&c = -2 \ae \ssc (\ae \varphi) \dn (\ae \varphi), \qquad &&\varphi \in \Big(-\frac{\K}{\ae}, \frac{\K}{\ae}\Big),\\
&\sinh \theta = \frac{2\ae^2 \Big(1-k^2\big(1+\cn^4 (\ae \varphi)\big)\Big)}{\alpha \cn^2(\ae \varphi)},  \qquad &&\cosh \theta = \frac{2 \ae^2 \Big(1 - k^2 \big(1 - \cn^4 (\ae \varphi)\big) \Big)}{\alpha \cn^2 (\ae \varphi)},
\end{align*}
where $\sn \psi, \cn \psi , \dn \psi$ are elliptic Jacobi functions with modulus $k$, and $\ds \ssc \psi= \frac{\sn \psi}{\cn \psi}$; $\ds \K (k) =\int_0^{\frac{\pi}{2}} \frac{ d t }{\sqrt{1-k^2 \sin^2 t}}$ is the complete elliptic integral of the first kind.

Immediate differentiation shows that in these coordinates the subsystem for the
costate variables~(\ref{dottht})--(\ref{dotat}) takes the following form:
\begin{align}
\dot{\varphi}=1, \qquad \dot{E} = 0, \qquad  \dot{\alpha}=0,
\end{align}
so that it has the solutions
\begin{align}
\varphi(t)=\varphi_t=\varphi_0+t, \qquad E = \const, \qquad \alpha=\const.
\end{align}

\subsection{Exponential mapping for timelike normal extremals }
Denote arguments of Jacobi functions $\psi_0 = \ae \varphi_0, \ \psi_t = \ae \varphi_t \in(-\K, \K)$ to describe the exponential mapping in general case $\alpha \neq 0$:
\begin{align}
x_1 (t) &= \frac{2 \ae}{|\alpha|} (\ssc \psi_0 \dn \psi_0 - \ssc \psi_t \dn \psi_t), \label{expx1}\\
x_2(t)&=\frac{4 \ae \Big(\ae (1-k^2) t - \big(\E(\psi_t)-\E(\psi_0)\big)\Big)-|\alpha|x_1(t)}{\alpha}. \\
y (t) &= -\frac{2 \ae^2}{\alpha |\alpha|}
   \Big(k^2 (\cn^2 \psi_t - \cn^2 \psi_0 ) + (1-k^2) (\nc^2 \psi_t-\nc^2 \psi_0)\Big)\nonumber\\
	&\qquad +\frac{\ae}{|\alpha|} \Big(\dn \psi_0 \ssc \psi_0 + \dn \psi_t \ssc \psi_t\Big) x_2(t), \label{expy} \\
z (t) &= \frac{\big(x_2(t)\big)^3}{6} + \frac{4 \ae^3}{3 \alpha^3} \bigg(2 \ae (k^2-1) t+(1-k^2) \Big(\frac{\dn \psi_t \sn \psi_t}{\cn^3 \psi_t}-\frac{\dn \psi_0 \sn \psi_0}{\cn^3 \psi_0}\Big)\nonumber\\
&\qquad+k^2 (\cn \psi_t \dn \psi_t \sn \psi_t-\cn \psi_0 \dn \psi_0 \sn \psi_0)-2 \big(\E(\psi_t)-\E(\psi_0)\big) (2 k^2-1)\bigg)\nonumber\\
&\qquad-\frac{2 \ae \dn \psi_0 \sn \psi_0 \Big(\frac{\ae \dn \psi_0 \ssc \psi_0}{|\alpha|} x_2(t) - \frac{1}{2} x_1(t) x_2(t)-y(t)\Big)}{\cn \psi_0 |\alpha|}\nonumber\\
&\qquad-\frac{2 \ae^2 (2 k^2-1) x_1(t)}{3 \alpha |\alpha|}, \label{expz}
\end{align}
where $\ds \E (\psi) = \int_0^\psi \dn^2 t\ d t$.

\begin{figure}[htbp]
\centering
\includegraphics[height=7cm]{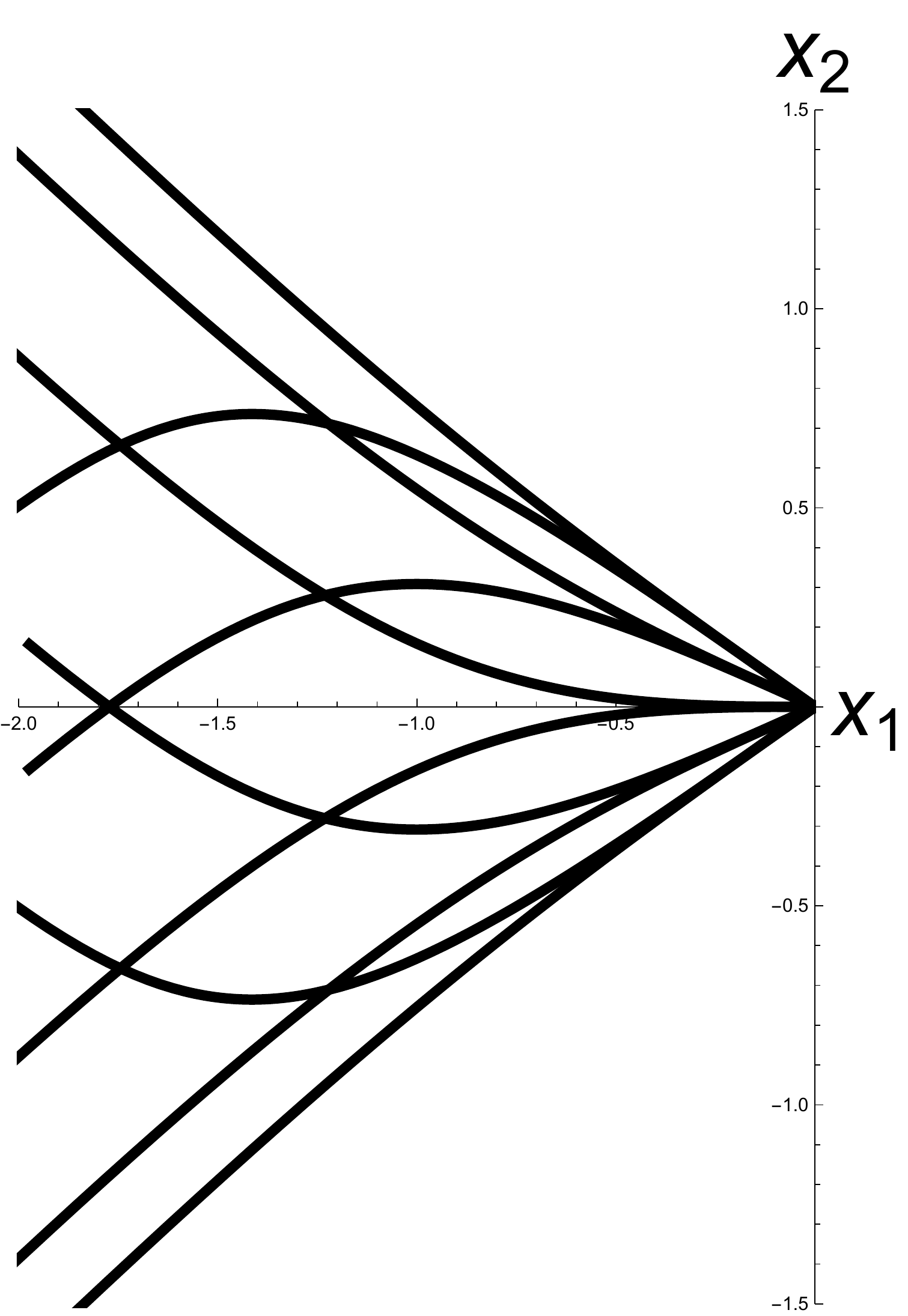} \quad \includegraphics[height=7cm]{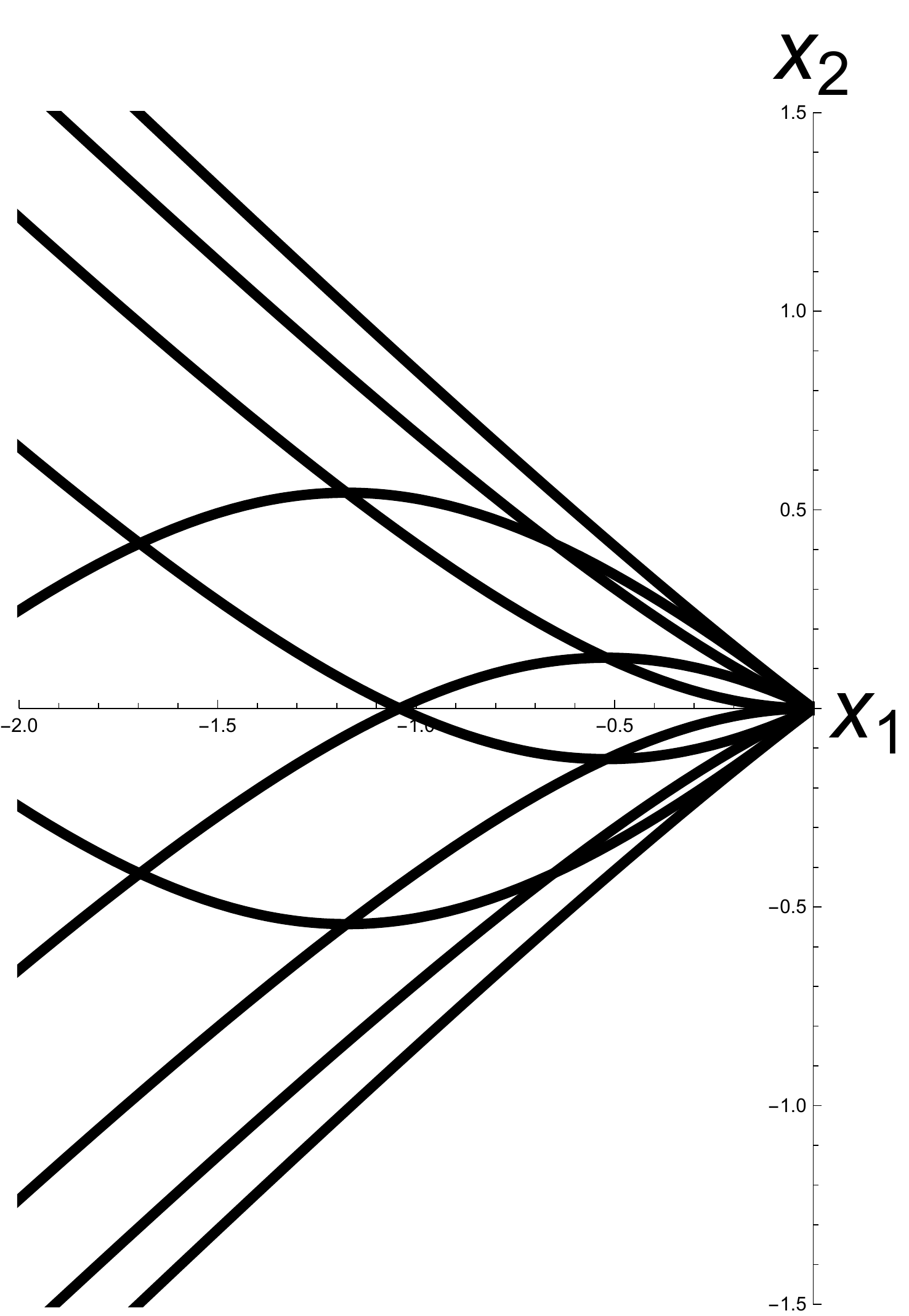} \quad \includegraphics[height=7cm]{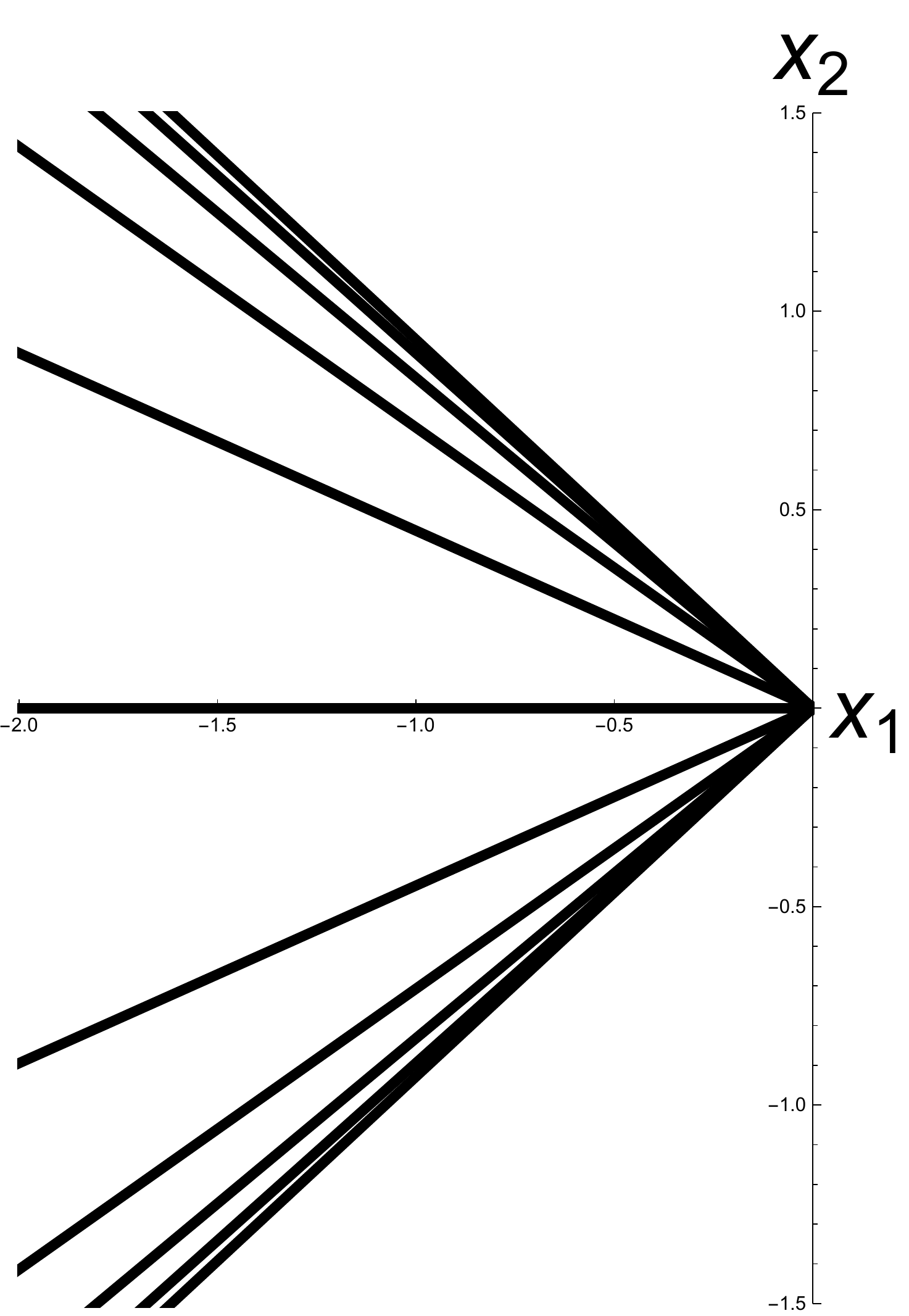}
\caption{Projections of extremals to the plane $(x_1, x_2)$ for $C^\pm$, $C^0$ and $C_0^0$ (timelike case)}\label{extr}
\end{figure}

Projections of extremals to the plane $(x_1, x_2)$ with $|\alpha|=1$ and $\psi_0 = 0$ are shown in Fig.~\ref{extr}, left. Notice that since $\psi_t \in(-\K, \K)$ we have an upper bound for the time parameter
\begin{align}
t < \tsupr (\lambda) = \frac{\K(k) - \psi_0}{\ae}. \label{supr}
\end{align}
When $t \to \tsupr (\lambda)$, solutions~(\ref{expx1})--(\ref{expz}) explode, i.e., tend to infinity.

For the subset $C^0$ we have
\begin{align*}
\alpha &= 0, \quad E = \frac{c^2}{2} ~ \Longrightarrow \quad  c \equiv \const, \ \theta = \theta_0 - c \, t, \\
x_1 (t) &= \frac{\sinh (\theta_0 - c \, t)-\sinh \theta_0}{c}, \\
x_2 (t) &= \frac{\cosh \theta_0  - \cosh (\theta_0 - c \, t)}{c}, \\
y (t) &= \frac{\sinh (c \, t) - c \ t}{2 c^2}, \\
z (t) &=  \frac{4 \sinh^3 \big(\frac{c \, t}{2}\big) \sinh \big(3 \theta_0 - \frac{3 c \, t}{2}\big)-3 \big(\sinh(c \, t) - c \, t\big) \sinh \theta_0}{6 c^3}.
\end{align*}
Projections of extremals to the plane $(x_1, x_2)$ with $|c|=1$ (hyperbolas) are shown in Fig.~\ref{extr}, center.

If $\lambda \in C_0^0$, then
\begin{align*}
\alpha &= c = E = 0, \quad \theta \equiv \const, \quad \sinh \theta = s_0,\quad \cosh\theta= c_0,\\
x_1 (t) &= - c_0 \ t, \\
x_2 (t) &= s_0 \ t, \\
y(t) &\equiv 0, \\
z (t) &= \frac{(2 s_0^2+1)s_0}{6} \ t^3.
\end{align*}
Projections of extremals to the plane $(x_1, x_2)$ are straight lines (see Fig.~\ref{extr}, right). For $\lambda \in C^0 \cup C_0^0$ we have no upper bound for the time parameter: $\tsupr(\lambda) = +\infty$.

So we obtained a parameterization of the exponential mapping
\begin{align}
&\Exp: N \to M = \R^4, \qquad \Exp (\lambda, t)=q(t),\\
&N = \{(\lambda, t) \in C \times \R_+ ~ | ~ t \in \big(0,\tsupr(\lambda)\big)\},
\end{align}
in the timelike case. It maps a pair $(\lambda, t)$ to the point of the corresponding extremal trajectory $q(t)$.

Further we investigate discrete symmetries of the exponential mapping and on this basis find estimates for the cut time on extremals.

\subsection{Discrete symmetries of exponential mapping}\label{symm}
Subsystem~(\ref{dottht})--(\ref{dotat}) for costate variables of the normal Hamiltonian system has symmetries which preserve the direction field of this system. We already described one of them~(\ref{tleps1}).

Let us define the action of symmetries $\varepsilon^i$ on the set of trajectories of the vertical subsystem with preservation of time direction. Denote a smooth curve $$\gamma = \Big\{\big(\theta (t), c (t), \alpha \big) ~ | ~ t \in [0,T]\Big\} \subset C.$$  Define the  action of symmetries on these curves (see. Fig.~\ref{reflimt}):
\begin{align*}
& \varepsilon^1 : \gamma \mapsto \gamma_1 = \Big\{\big(\theta^1(t), c^1(t), \alpha^1\big) ~ | ~ t \in [0, T]\Big\} = \Big\{\big(-\theta (t), -c(t), -\alpha\big)\Big\}, \\
& \varepsilon^2 : \gamma \mapsto \gamma_2 = \Big\{\big(\theta^2(t), c^2(t), \alpha^2\big) ~ | ~ t \in [0, T]\Big\} = \Big\{\big(\theta(T-t), -c(T-t), \alpha\big)\Big\}, \\
& \varepsilon^3 : \gamma \mapsto \gamma_3 = \Big\{\big(\theta^3(t), c^3(t), \alpha^3\big) ~ | ~ t \in [0, T]\Big\} = \Big\{\big(-\theta(T-t),c (T-t), -\alpha\big)\Big\}.
\end{align*}
It is obvious that if $\gamma$ is a solution to the vertical subsystem~(\ref{dottht})--(\ref{dotat}), then $\gamma_i, i = 1,2,3$, is a solution as well.

Now consider the group of symmetries $G = \{ \mathrm{Id}, \varepsilon^1, \varepsilon^2, \varepsilon^3 = \varepsilon^1 \circ \varepsilon^2\} \cong \Z_2 \times \Z_2$. The symmetry $\varepsilon^1$ preserves the direction of time, while $\varepsilon^2$ and $\varepsilon^3$ change it.

\begin{figure}[htbp]
\centering
\includegraphics[width=0.47\linewidth]{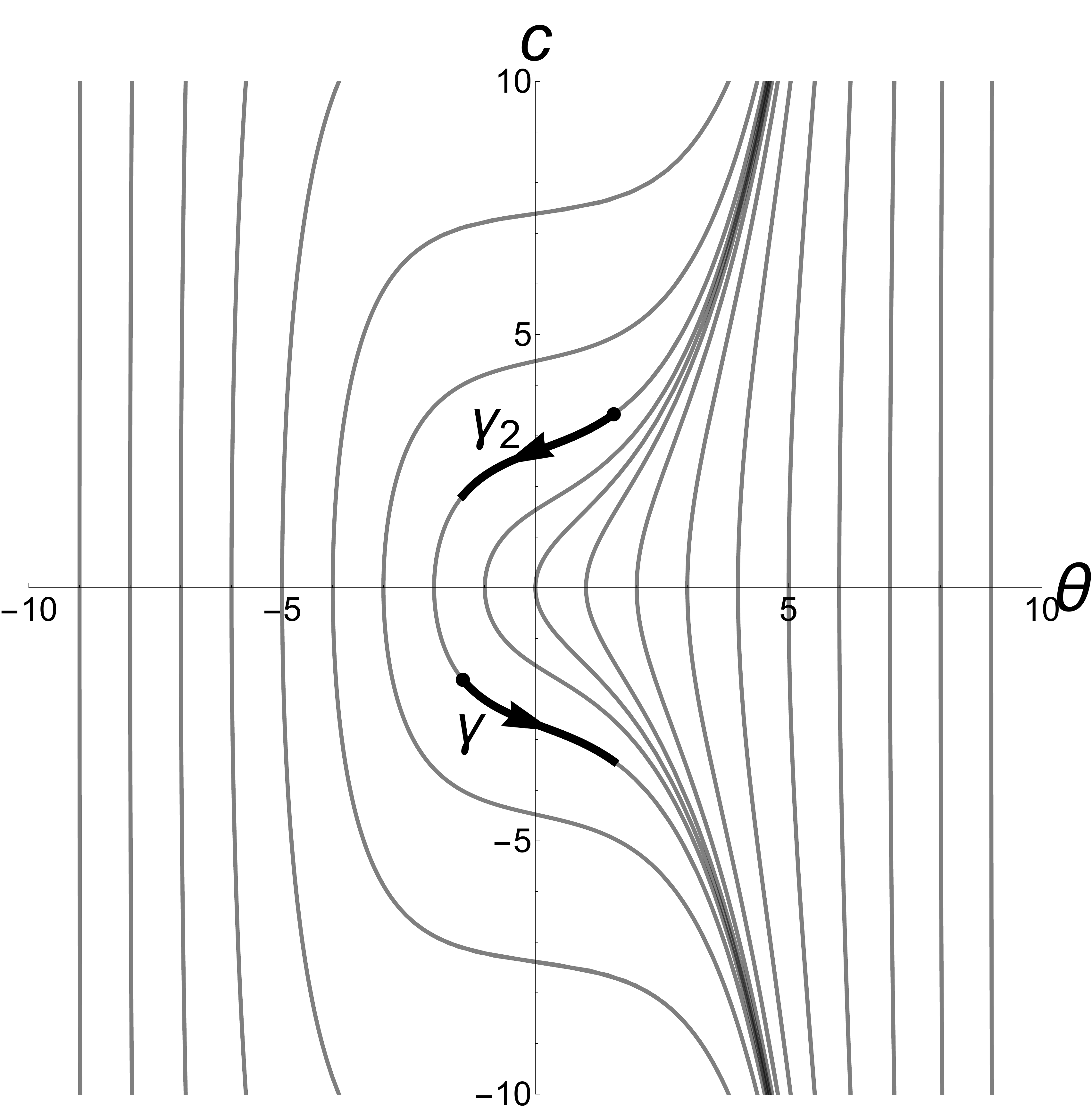}
\qquad
\includegraphics[width=0.47\linewidth]{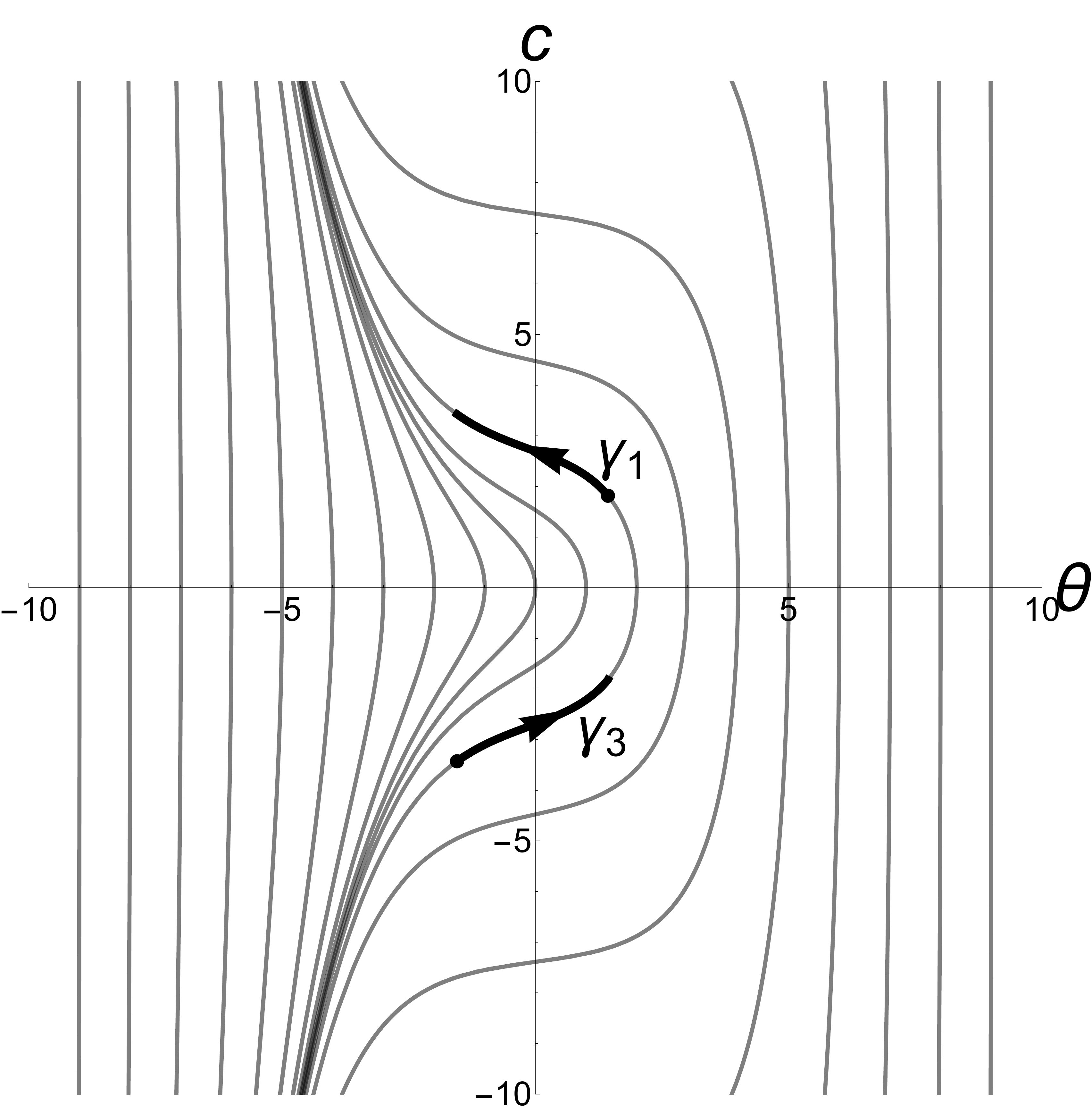}
\caption{Action of symmetries on trajectories of vertical subsystem (timelike case)}\label{reflimt}
\end{figure}

Continue the action of symmetries from the vertical subsystem to solutions of the Hamiltonian system of PMP
\begin{align}
&\dot{\theta} (t) = - c (t), \quad \dot{c}(t) = - \alpha \cosh \theta (t), \quad \dot{\alpha} = 0, \label{ham3}\\
&\dot{q} (t) = -\cosh \theta (t) X_1 \big(q(t)\big) + \sinh \theta (t) X_2 \big(q (t)\big), \label{ham4}
\end{align}
in the following way:
\begin{align}
\varepsilon^i: \Big\{\big(\theta (t), c (t), \alpha, q(t)\big) ~ | ~ t \in [0, T]\Big\} \mapsto \Big\{ \big(\theta^i(t), c^i(t), \alpha^i, q^i(t)\big) | t \in [0,T]\Big\}, \label{symham}
\end{align}
where $q(t) = \big(x_1 (t), x_2 (t), y(t), z(t)\big), \ t \in [0,T],$ is a geodesic and
\begin{align*}
q^i(t) = \big( x_1^i(t), x_2^i (t), y^i (t), z^i(t)\big), \qquad t \in [0,T], \ i = 1, 2, 3,
\end{align*}
are its images under the action of the symmetries $\varepsilon^i$.

\begin{lemma}\label{lem1t} The symmetries $\varepsilon^i$ map trajectories on the plane $(x_1, x_2)$ in the following way:
\begin{align*}
&x_1^1(t) = x_1(t), \qquad &&x_2^1(t) = - x_2(t), \\
&x_1^2(t) =  x_1(T) - x_1(T-t), \qquad &&x_2^2(t) = x_2(T)  - x_2 (T-t), \\
&x_1^3(t) = x_1(T)  - x_1(T-t), \qquad &&x_2^3(t) = x_2(T-t)- x_2(T).
\end{align*}
\end{lemma}
\begin{proof}
Via direct integration, for instance, for $\varepsilon^2$ we verify that
\begin{align*}
&x_1^2(t) = \int_0^t \big(-\cosh \theta(T-s)\big) d s = \int_{T-t}^{T} \big(-\cosh \theta(r)\big) d r = x_1(T) - x_1(T-t), \\
&x_2^2(t) = \int_0^t \sinh \theta (T-s) d s = \int_{T-t}^{T} \sinh \theta(r) d r = x_2(T)  - x_2 (T-t).
\end{align*}
\end{proof}

\begin{lemma}\label{lem2t} The symmetries $\varepsilon^i$ map endpoints of geodesics $q = (x_1, x_2, y, z)$ to endpoints of $q^i = (x_1^i, x_2^i, y^i , z^i\big)$ in the following way:
\begin{align*}
& x_1^1 (T)= x_1(T),  &&x_2^1 (T) = -x_2(T), &&y^1(T) = -y(T), &&z^1(T) = -z(T), \\
& x_1^2 (T)= x_1(T), &&x_2^2(T) = x_2(T), &&y^2(T) = -y(T), &&z^2(T) = z(T)-x_1(T) y(T), \\
& x_1^3 (T)= x_1(T), &&x_2^3(T) = -x_2(T), &&y^3(T) = y(T), &&z^3(T) = -z(T) + x_1(T) y(T).
\end{align*}
\end{lemma}
\begin{proof}
Lemma~\ref{lem1t} gives us the expressions for $x_1^i (T)$ and $x_2^i (T)$. The expressions for other variables are obtained by integration. For example, for $\varepsilon^2$ we have
\begin{align*}
y^2 (T) &= \int_0^T \frac{x_2^2(t) \cosh \theta^2 (t) + x_1^2 (t) \sinh \theta^2 (t)}{2} d t= \frac{x_2(T)}{2}\int_0^T \cos \theta(T-t) d t \\
&\qquad\quad+ \frac{x_1(T)}{2}\int_0^T \sinh \theta(T-t) d t - \int_0^T \frac{x_2(s) \cosh \theta (s)  + x_1(s) \sinh \theta (s)}{2} d s  \\
	&\qquad= \frac{x_2(T) x_1(T)}{2} - \frac{x_1(T) x_2(T)}{2} - y(T)= -y(T), \\
z^2(T) &= \int_0^T \frac{\big(x_1^2(t)\big)^2 + \big(x_2^2(t)\big)^2}{2}\dot{x}_2^2(t) d t = \frac{\big(x_2^2(T)\big)^3}{6}+ \int_0^T \frac{\big(x_1^2(t)\big)^2}{2} \dot{x}_2^2 (t) d t  \\
&\qquad = \frac{\big(x_2(T)\big)^3}{2} + \int_0^T \frac{\big(x_1(T)\big)^2 - 2 x_1 (T) x_1(T-t) + \big(x_1 (T-t)\big)^2}{2} \dot{x}_2 (T-t) d t \\
&\qquad = \frac{\big(x_2 (T)\big)^3}{6} + \frac{\big(x_1 (T)\big)^2 x_2 (T)}{2} - x_1 (T) \int_0^T x_1 (s) \dot{x}_2 (s) d s + z(T) - \frac{\big(x_2 (T)\big)^3}{6} \\
&\qquad = \frac{\big(x_1 (T)\big)^2 x_2 (T)}{2} - x_1 (T) \Big(y(T) + \frac{x_1(T)x_2(T)}{2}\Big) + z(T) = z(T) -  x_1 (T) y(T).
\end{align*}
\end{proof}
We define the action of $\varepsilon^i$ in the preimage $N$ of the exponential mapping by restricting the action to initial point of trajectory of vertical subsystem:
\begin{align}
&\varepsilon^i: C \to C, \qquad \varepsilon^i (\theta, c, \alpha) = (\theta^i, c^i, \alpha^i), \label{defpre1}\\
&(\theta^1, c^1, \alpha^1) = (-\theta, -c , -\alpha), \\
&(\theta^2, c^2, \alpha^2) = (\theta, -c, \alpha), \\
&(\theta^3, c^3, \alpha^3) = (-\theta, c, -\alpha), \label{defpre2}
\end{align}
in the following way:
\begin{align*}
&\varepsilon^1(\lambda, t) = \big(\varepsilon^1 (\lambda), t\big),  \\
&\varepsilon^i(\lambda, t) = \big(\varepsilon^i \circ e^{t \vec{H}_v} (\lambda), t\big),  \qquad i = 2, 3,
\end{align*}
where $\vec{H}_v = -c \frac{\partial}{\partial \theta} - \alpha \cosh \theta \frac{\partial}{\partial c} \in \mathrm{Vec}(C)$ is the vertical part of Hamiltonian vector field.

We define the action of $\varepsilon^i$ in the image of exponential mapping $M$ by restricting the action to endpoints of extremals (see Lemma~\ref{lem2t}):
\begin{align}
&\varepsilon^i : M \to M, \qquad \varepsilon^i (q) = \varepsilon^i (x_1, x_2, y, z) = q^i = (x_1^i, x_2^i, y^i, z^i), \label{defim1}\\
&(x_1^1, x_2^1, y^1, z^1) = (x_1, \ -x_2, \ -y, \ -z), \\
&(x_1^2, x_2^2, y^2, z^2) = (x_1, \  x_2, \ -y, \ z - x_1 y), \\
&(x_1^3, x_2^3, y^3, z^3) = (x_1, \ -x_2, \ y, \ x_1 y - z).\label{defim2}
\end{align}

Since the action of $\varepsilon^i$ in the domain $N$ and the image $M$ of the exponential map are induced by the actions of symmetries~(\ref{symham}) on trajectories of the Hamiltonian system~(\ref{ham3})--(\ref{ham4}), we have the following result.

\begin{proposition}\label{prop1t}
The mappings $\varepsilon^i, i= 1, 2, 3,$ are symmetries of the exponential mapping, i.e.,
\begin{align*}
&(\varepsilon^i \circ \Exp) (\theta,c,\alpha,t) = (\Exp \circ \, \varepsilon^i) (\theta, c, \alpha, t), \qquad (\theta, c, \alpha, t) \in N.
\end{align*}
\end{proposition}

\subsection{Maxwell points of timelike normal extremals}
A point $q(t)$ on a geodesic is called a Maxwell point if there exists another geodesic $\tilde{q}(s)\not \equiv q(s)$ for which $\tilde{q}(t) = q(t), t>0$. It is known that a geodesic cannot be optimal after a Maxwell point. In this section we compute the Maxwell points corresponding to some symmetries $\varepsilon^i$. On this basis we derive estimates for the cut time along geodesics
$$\tcut (\lambda) = \sup \{ t>0 ~ | ~ \Exp(\lambda, s) \text{ is optimal for } s \in[0,t]\}.$$

We define Maxwell sets in the preimage of $\Exp$ corresponding to the symmetries $\varepsilon^i$:
\begin{align}
\MAX^i &= \big\{(\lambda,t) \in N ~ | ~ \lambda^i \neq \lambda, \ \Exp(\lambda^i, t) = \Exp(\lambda, t) \big\}, \label{defmax}\\
\lambda &= (\theta, c, \alpha), \qquad \lambda^i = \varepsilon^i(\lambda).\nonumber
\end{align}

It follows from Proposition~\ref{prop1t} that the equality $\Exp (\lambda^i, t) = \Exp (\lambda, t)$ is equivalent to $\varepsilon^i\big(q(t)\big) = q(t)$. Therefore we get  a description of fixed points of the symmetries $\varepsilon^i$ in the image of the exponential mapping.

\begin{lemma}\label{lemim}~

\begin{enumerate}
\item $\varepsilon^1(q) = q \quad \Longleftrightarrow \quad  x_2 =0,\ y =0,\ z= 0$.
\item $\varepsilon^2(q) = q \quad \Longleftrightarrow \quad  y = 0$.
\item $\ds \varepsilon^3(q) = q \quad \Longleftrightarrow \quad  x_2=0,\ z  = \frac{x_1 y}{2}$.
\end{enumerate}
\end{lemma}
\begin{proof}
Follows from definition~(\ref{defim1})--(\ref{defim2}) of action of the symmetries in $E$.
\end{proof}

Now we compute fixed points of symmetries in the preimage of exponential mapping, which is necessary for description of Maxwell sets. Introduce the following coordinates in the sets $N^\pm = \{(\lambda, t) \in C^\pm \times \R_+ \mid t \in \big(0, \tsupr (\lambda)\big)\}$:
\begin{align}
&\tau = \frac{\psi_t + \psi_0}{2}, \qquad p = \frac{\psi_t-\psi_0}{2} = \frac{\ae t}{2}. \label{ptaut}
\end{align}
The parameter $\tau$ corresponds to the middle point of a trajectory.

\begin{lemma}\label{lempre}~

\begin{enumerate}
\item\label{item1} $\lambda^1 = \lambda \quad \Longleftrightarrow \quad \lambda \in C_0^0,\ \theta = 0.$
\item $\lambda^2 = \lambda \quad \Longleftrightarrow \quad \lambda \in C^\pm, \tau = 0$ \quad or \quad $\lambda \in C_0^0$.
\item $\ds \lambda^3 = \lambda \quad \Longleftrightarrow \quad \lambda \in C^0 \cup C_0^0,\ \theta + \frac{c t}{2}  = 0$
\end{enumerate}
\end{lemma}
\begin{proof}
Follows from the definition of action of reflections in the domain of exponential mapping~(\ref{defpre1})--(\ref{defpre2}).
\end{proof}

\begin{theorem}\label{thrm}
Let $(\lambda,T) \in N$ and $q(T) = \Exp (\lambda, T)$. Then
\begin{enumerate}
\item $(\lambda, T) \in \MAX^1 \quad \Longleftrightarrow \quad x_2(T) = y(T) = z(T) = 0$ \ and \ if $\lambda \in C_0^0$, then $\theta \neq 0$.
\item $(\lambda, T) \in \MAX^2 \quad \Longleftrightarrow \quad \lambda \in C^\pm, y(T) = 0, \tau \neq 0$ \ or \ $\lambda \in C^0, y(T) = 0$.
\item $(\lambda, T) \in \MAX^3 \quad \Longleftrightarrow \quad x_2 = 0, \ds z = \frac{x_1 y}{2}$ and if $\lambda \in C^0\cup C_0^0$, then $\ds \theta \neq - \frac{c t}{2}$.
\end{enumerate}
\end{theorem}
\begin{proof}
Follows from definition~(\ref{defmax}) of Maxwell sets, Lemmas~\ref{lemim} and~\ref{lempre}.
\end{proof}

The equations $\varepsilon^i (q) = q$ \ define sub-manifolds in $\R^4$ with dimension from 3 to 1, exponential mapping translates the corresponding Maxwell sets into these sub-manifolds:
$$\Exp (\MAX^i) \subset \{ q \in M ~ | ~ \varepsilon^i (q) = q\}, \qquad i = 1, 2 ,3.$$

Further we describe in details the set $\MAX^2$ since it defines a sub-manifold in $M$ with the maximum dimension 3, while $\MAX^1$ and $\MAX^3$ define sub-manifolds with dimension 1 and 2. Using new coordinates (\ref{ptaut}) in the case $\alpha \neq 0$ we get
\begin{align*}
y &= -\frac{\sn \tau \cn \tau \dn \tau f_y(p)}{\alpha |\alpha| \ae^2 k^2 (\cn^2 \tau - \dn^2 \tau \sn^2 p) (1 - k^2 \sn^2 p \sn^2 \tau)},\\
f_y (p) &= -16 \ae^4 k^4 \cn^3 p \dn p \sn p + (16 \ae^4  k^2 \E (p) - \alpha^2 p) (\dn^2 p - k^2 \cn^2 p \sn^2 p).
\end{align*}

This gives us a description of the Maxwell set:
\begin{align}
\MAX^2 \cap N^\pm &= \{\nu \in N^\pm ~ | ~ y(\nu) = 0, \tau \neq 0, \tau \in (-\K, \K), p \in (0, \K) \}  \nonumber\\
 &=\{\nu \in N^\pm ~ | ~ f_y (p) = 0,  p \in (0, \K)\}, \label{max2t}
\end{align}
where $\nu = (\tau, p, E, \alpha)$.

The following obvious lemma will be useful for localization of roots of functions.

\begin{lemma} \label{lem4}
Let smooth functions $f(u), g(u)$ satisfy on $(0, u_0) \subset \R$ the conditions
\begin{align}
&f(u) \not \equiv 0, \quad g(u) > 0, \quad \bigg(\frac{f(u)}{g(u)}\bigg)' \geq 0, \label{cond1}\\
&\lim_{u\to 0} \frac{f(u)}{g(u)} = 0. \label{cond2}
\end{align}
Then $f(u) > 0$ for $u \in (0, u_0)$.
\end{lemma}

If functions $f$ and $g$ satisfy conditions (\ref{cond1}), (\ref{cond2}), then we say that $g$ is a comparison function for $f$ on the interval $(0,u_0)$.

\begin{lemma} \label{lem5t}
The function $f_y(p)>0$ for $p \in \big(0,\K(k)\big), \ k \in (0, 1)$.
\end{lemma}
\begin{proof}
We show that the function $g(p) = k^2  \big(1 - k^2 (1 - \cn^4 p)\big)$ is a comparison function for $f_y(p)$ for $p \in (0,\K)$.

The inequality $f_y(p)\not \equiv 0$ follows from the expansion $f_y(p) = \frac{4}{3}\alpha^2 k^2 p^3 + o(p^3)$.
Notice that $g(p) > 0$ for $p \in (0, \K)$. Finally we get the equalities
\begin{align*}
& \bigg(\frac{f_y(p)}{g(p)}\bigg)' = \frac{4 \alpha^2 \sn^2 p \cn^2 p \dn p}{1 - k^2 (1 - \cn^4 p)}>0, \quad \frac{f_y(p)}{g(p)} = \frac{4}{3} \alpha^2 p^3  + o(p^3).
\end{align*}
So $g(p)$ is a comparison function for $f_y (p)$; thus, it follows from Lemma~\ref{lem4} that $f_y (p) > 0$ for $p \in  (0, \K)$.
\end{proof}

\begin{theorem}
$\MAX^2 \cap N^\pm = \emptyset.$
\end{theorem}
\begin{proof}
Follows immediately from description of Maxwell set~(\ref{max2t}) and Lemma~\ref{lem5t}.
\end{proof}

We proved that a general geodesic with $\lambda \in C^\pm$ has no Maxwell times $\tmax(\lambda) \in \big(0, \tsupr(\lambda)\big)$ corresponding to the symmetries $\varepsilon^i, i = 1, 2, 3$. So we conjecture that geodesics $q(t) = \Exp(\lambda, t)$ with $\lambda \in C^\pm$ are optimal for $t\in \big(0, \tsupr(\lambda)\big)$.

\section{Spacelike normal extremal trajectories}
In the spacelike case $H=\frac{1}{2}\big(-h_1^2+h_2^2\big) > 0,$  we consider extremals on the level surface $H= 1/2$ and introduce new coordinates on this surface:
\begin{align*}
&h_1 = \sinh \theta,\quad \ h_2 = \pm \cosh \theta, \quad h_3=c, \quad \ h_4 = \alpha.
\end{align*}

Notice that we have the following symmetry of the Hamiltonian system:
\begin{align}
\varepsilon^0 : (h_1, h_2, h_3, h_4, x_1, x_2, y, z) \mapsto (h_1, -h_2, -h_3, -h_4, x_1, -x_2, -y, -z),
\end{align}
So we assume  $h_2 = \cosh \theta > 0$ without loss of generality in the sequel.

In the variables $(\theta, c, \alpha, x_1, x_2, y, z)$ on the level surface $\{H = 1/2\}$ the Hamiltonian system of Pontryagin maximum principle takes the following form in the normal case for spacelike curves:
\begin{align}
&\dot{x}_1 = -\sinh \theta, \label{dotx1}\\
&\dot{x}_2 = \cosh \theta, \label{dotx2}\\
&\dot{y} = \frac{x_2 \sinh \theta + x_1 \cosh \theta}{2}, \label{doty}\\
&\dot{z} = \frac{x_1^2+x_2^2}{2}\cosh \theta , \label{dotz}\\
&\dot{\theta} = -c, \label{dotth}\\
&\dot{c} = -\alpha \sinh \theta, \label{dotc}\\
&\dot{\alpha} = 0. \label{dota}
\end{align}

\begin{figure}[htbp]
\centering
\includegraphics[width=0.47\linewidth]{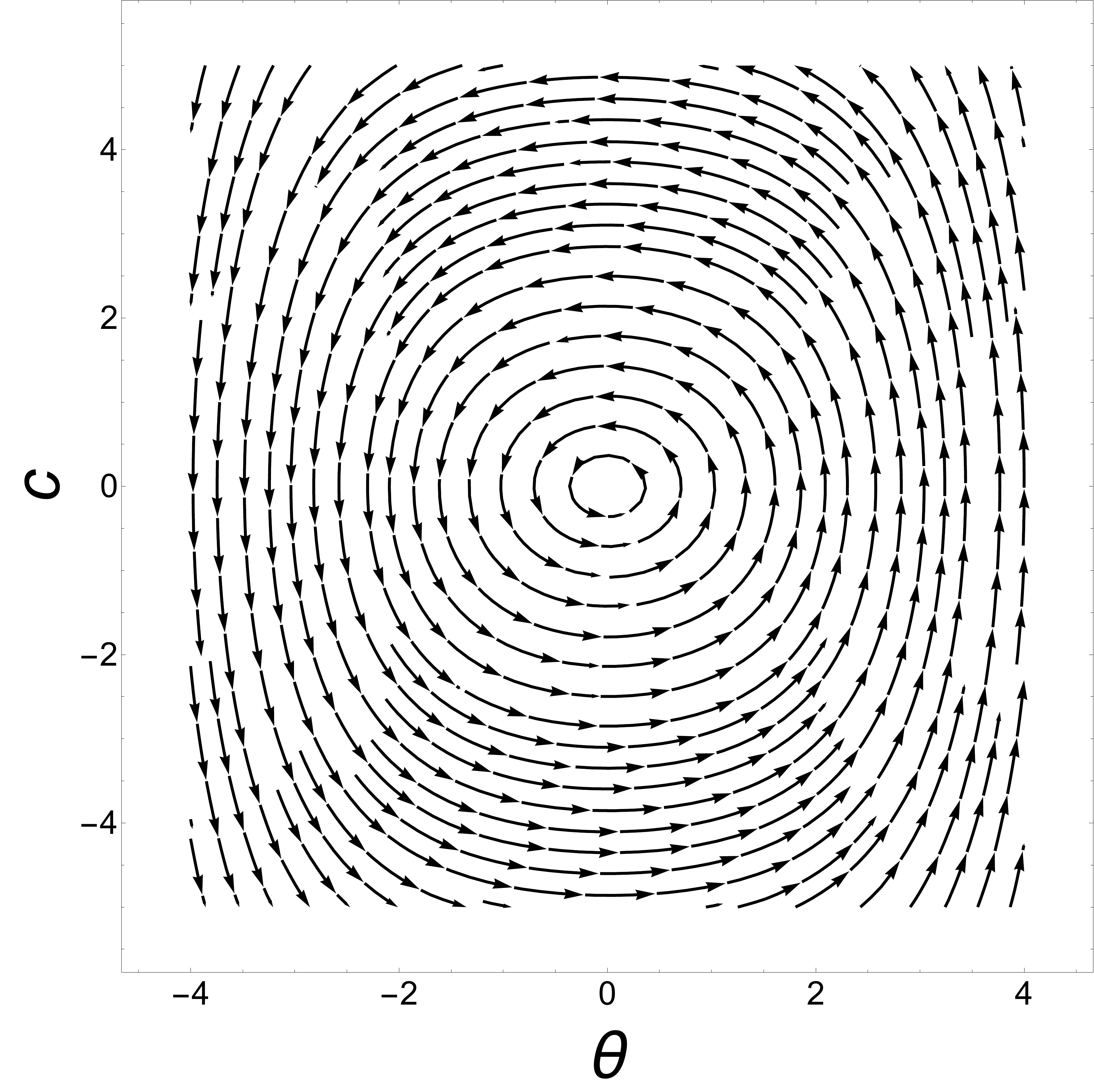}
\qquad
\includegraphics[width=0.47\linewidth]{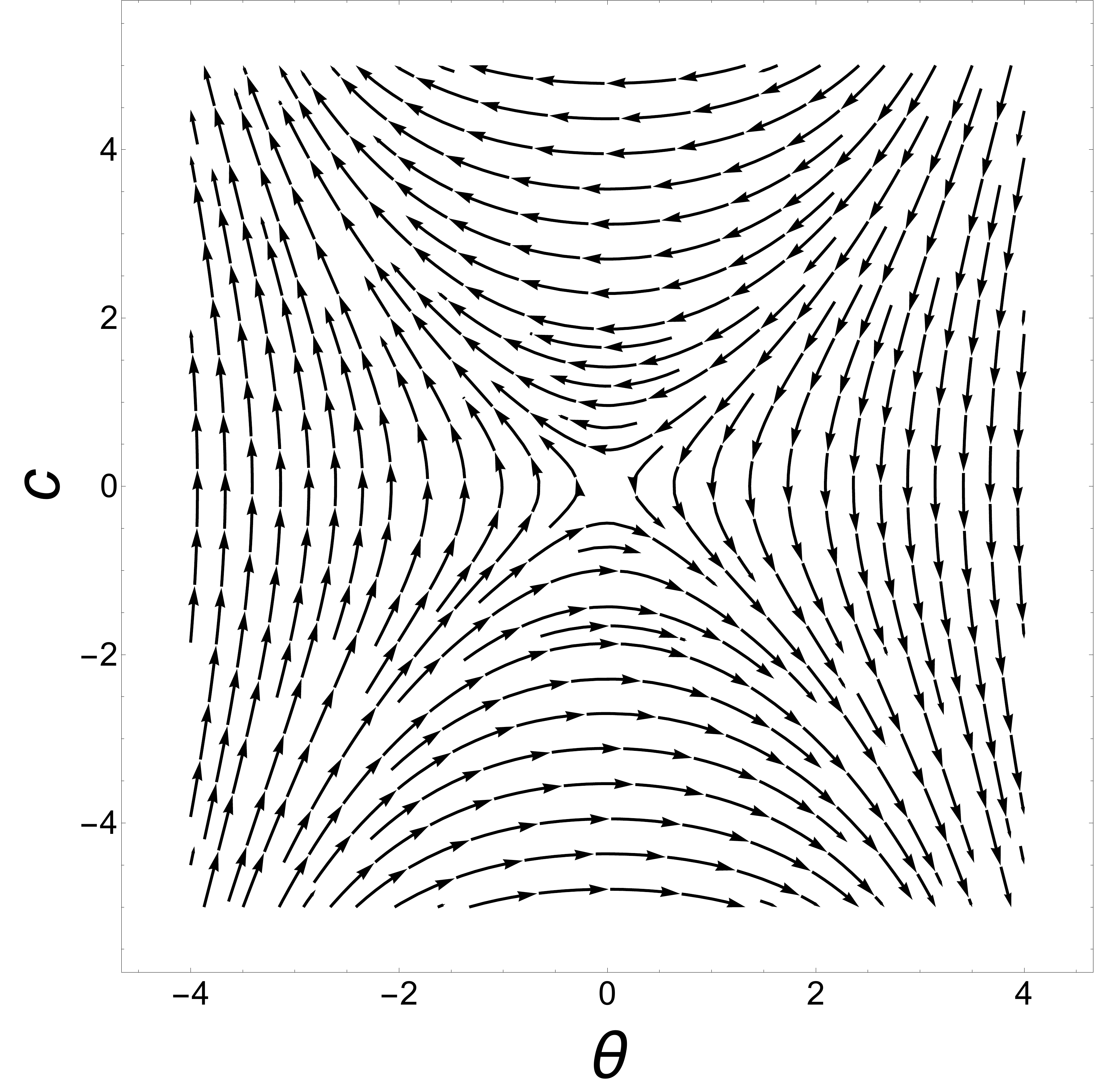}
\caption{Phase portraits of the vertical subsystem for $\alpha = -1, \ 1$ (spacelike case)}\label{portrait+-}

\end{figure}
The vertical subsystem~(\ref{dotth})--(\ref{dota}) has the energy integral:
\begin{align}
&E = \frac{h_3^2}{2} - h_2 h_4 = \frac{c^2}{2} - \alpha \cosh \theta.
\end{align}
The family of all normal extremals is parametrized by points of the set

\begin{align*}
C&=T_{q_0}^* \cap \{H = 1/2, h_2 > 0\} = \big\{(h_1, h_2, h_3, h_4) \in \R^4 ~ | ~ h_2^2-h_1^2=1, h_2>0\big\}\\
&=\big\{(\theta, c, \alpha) \in \R^3\big\}.
\end{align*}

The set $C$ has the following decomposition into invariant sets of equations~(\ref{dotth})--(\ref{dota}):

\begin{align*}
C &= \cup_{i=1}^7 C_i,   \quad  C_i \cap C_j = \emptyset, \ i \neq j,  \qquad \lambda = (\theta, c, \alpha), \\
C_1&= \{\lambda \in C ~ | ~ \alpha < 0, E >-\alpha\}, \\
C_2&=\{\lambda \in C ~ | ~ \alpha > 0, E < -\alpha\}, \\
C_3&=\{\lambda \in C ~ | ~ \alpha > 0, E > - \alpha\}, \\
C_4&=\{\lambda \in C ~ | ~ \alpha > 0, E = -\alpha, \ c\neq 0, \theta \neq 0\}, \\
C_5&=\{\lambda \in C ~ | ~ \alpha \neq  0, E = - \alpha, \ c = 0, \theta = 0\}, \\
C_6&=\{\lambda \in C ~ | ~ \alpha = 0, \ E > 0, c \neq 0\},\\
C_7&=\{\lambda \in C ~ | ~ \alpha = 0, \ E = 0, c = 0\}.
\end{align*}

In order to parameterize extremals, we introduce coordinates $(\varphi, E, \alpha)$ on the set $\bigcup_{i=1}^3 C_i$ in the following way.

\begin{figure}[htbp]
\centering
\includegraphics[width=0.47\linewidth]{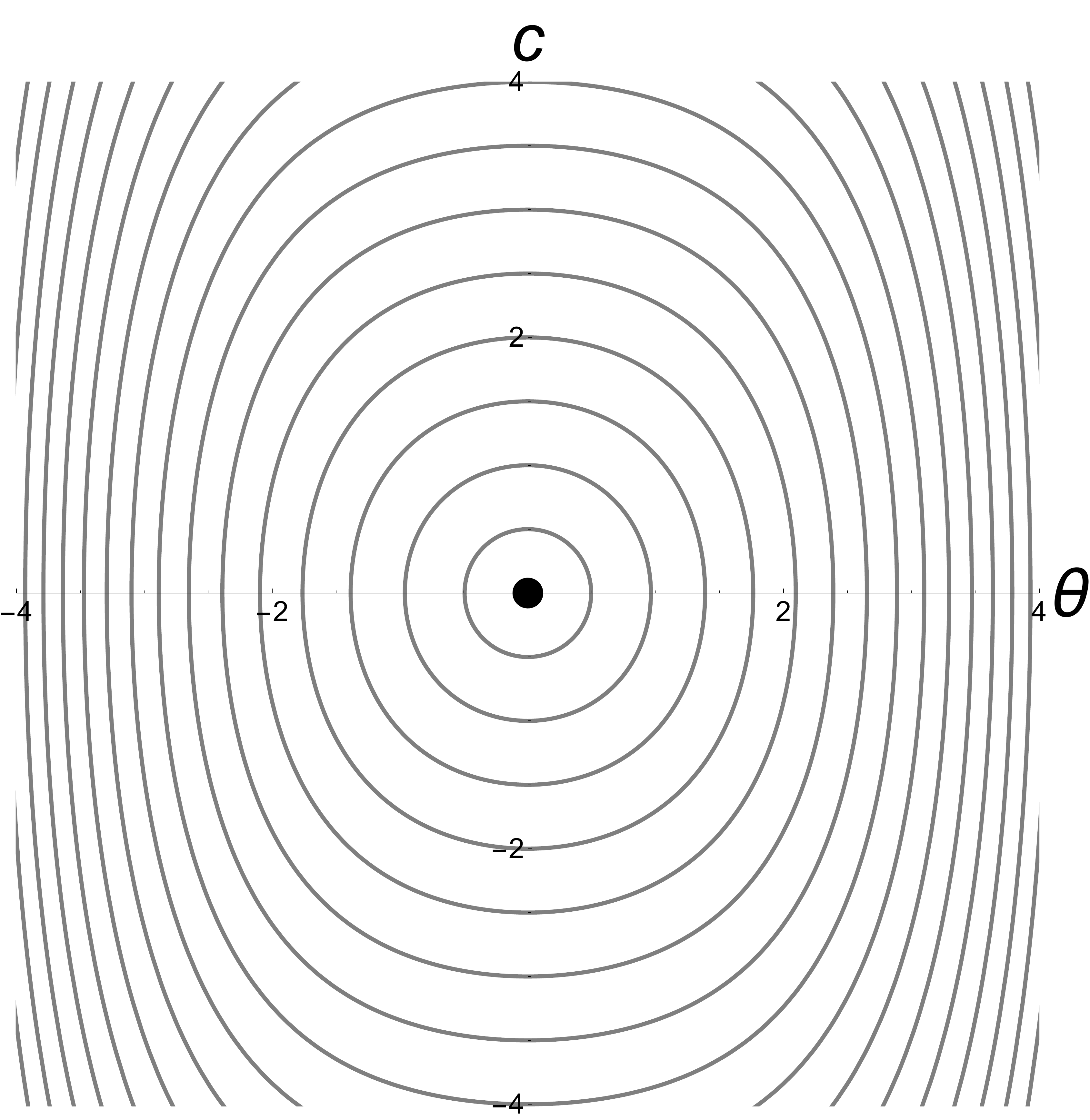}
\qquad
\includegraphics[width=0.47\linewidth]{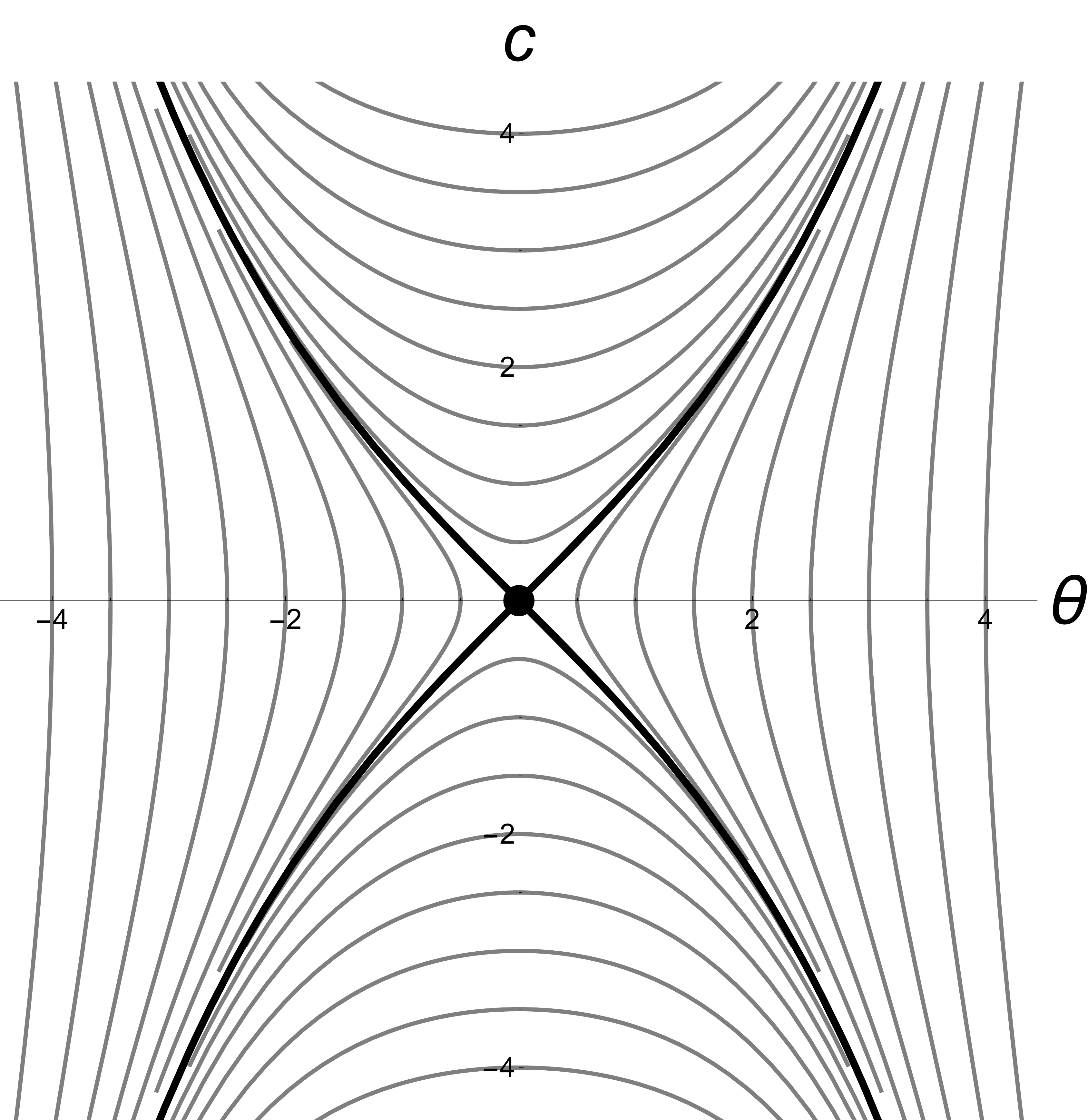}
\caption{Trajectories of the vertical subsystem for $\alpha = -1, \ 1$ (spacelike case)}\label{a+-}
\end{figure}

In the domain $C_1$:
\begin{align*}
k&=\sqrt{\frac{E+\alpha}{E-\alpha}}\in(0,1), \qquad &&\ae=\sqrt{\frac{E-\alpha}{2}}, \\
c&= \sqrt{2 (E+\alpha)} \sn (\ae \varphi), \qquad &&\varphi \in \R, \\
\sinh \theta &= \frac{\sqrt{E-\alpha} \ \sqrt{E+\alpha} \cn (\ae \varphi) \dn (\ae \varphi)}{\alpha}, \qquad &&\cosh \theta = 1-\frac{(E+\alpha) \cn^2 (\ae \varphi)}{\alpha}.
\end{align*}

In the domain $C_2$:
\begin{align*}
k&=\sqrt{\frac{2 \alpha}{\alpha-E}}\in(0,1), \qquad &&\ae = \sqrt{\frac{\alpha-E}{2}},\\
c&=\sgn \theta \frac{\sqrt{2(-\alpha-E)} \sn(\ae \varphi)}{\cn(\ae \varphi)}, \qquad &&\varphi \in \Big(-\frac{\K}{\ae}, \frac{\K}{\ae}\Big),\\
\sinh \theta &= \sgn \theta \frac{\sqrt{-\alpha-E} \ \sqrt{\alpha-E} \dn(\ae \varphi)}{\alpha \cn^2(\ae \varphi)}, \qquad
&&\cosh \theta = 1 - \frac{\alpha+E}{\alpha \cn(\ae \varphi)}.
\end{align*}

In the domain $C_3$:
\begin{align*}
k&=\sqrt{\frac{E-\alpha}{E+\alpha}}\quad k_2 = k^2\in(-\infty, 1), \qquad &&\ae = \sqrt{\frac{E+\alpha}{2}}, \\
c&=\sgn c \frac{2 \ae \dn (\ae \varphi)}{\cn(\ae \varphi)}, \qquad &&\varphi \in \Big(-\frac{\K}{\ae}, \frac{\K}{\ae}\Big),\\
\sinh \theta &= \sgn c \frac{2 \sn (\ae \varphi)}{\cn^2(\ae \varphi)}, \qquad && \cosh \theta = \frac{1+\sn^2(\ae \varphi)}{\cn^2(\ae \varphi)}.
\end{align*}

\begin{remark}
If $k^2<0$, then $\ds k = i \frac{\tilde{k}}{\sqrt{1-\tilde{k}^2}},\tilde{k}=\sqrt{\frac{-k^2}{1-k^2}}\in (0,1).$ Using formulas from $ \cite{ahiez} $ we have
\begin{align}
\tilde{k} &= \sqrt{\frac{\alpha - E}{2 \alpha}}, \quad \sqrt{1- \tilde{k}^2} = \sqrt{\frac{\alpha+E}{2 \alpha}}; \quad &\dn (\ae \varphi, k) &= \frac{1}{\dn(\sqrt{\alpha} \varphi, \tilde{k})},\\
\sn (\ae \varphi, k) &= \frac{\sqrt{1-\tilde{k}^2} \sn (\sqrt{\alpha} \varphi, \tilde{k})}{\dn(\sqrt{\alpha} \varphi, \tilde{k})},
  \qquad &\cn (\ae \varphi, k) &= \frac{\cn (\sqrt{\alpha} \varphi, \tilde{k})}{\dn (\sqrt{\alpha} \varphi, \tilde{k})}.
\end{align}
\end{remark}

In the domain $C_4$:
\begin{align*}
c&= \sgn \theta \frac{4 \alpha e^{\ae \varphi}}{\alpha e^{2 \ae \varphi}-1}, \qquad && \ae = \sqrt{\alpha}, \quad \varphi \in \R, \ \alpha e^{2 \ae \varphi} \neq 1, \\
\sinh \theta &= \sgn \theta \frac{4 \ae e^{\ae \varphi} (\alpha e^{2 \ae \varphi}+1)}{(\alpha e^{2 \ae \varphi}-1)^2}, \qquad &&\cos \theta = \frac{\alpha e^{2 \ae \varphi}(\alpha e^{2 \ae \varphi}+6)+1}{\alpha e^{2 \ae \varphi}-1}.
\end{align*}

Immediate differentiation shows that in the coordinates $(\varphi, E, \alpha)$ the subsystem for the costate variables~(\ref{dotth})--(\ref{dota}) takes the following form:
\begin{align}
\dot{\varphi} = 1, \qquad \dot{E}=0, \qquad \dot{\alpha}=0,
\end{align}
so that it has solutions
\begin{align}
\varphi(t) = \varphi_t = \varphi_0+t, \qquad E=\const, \qquad \alpha = \const.
\end{align}

\subsection{Exponential mapping for spacelike normal extremals}
Denote $\psi_0 = \ae \varphi_0, \psi_t = \ae \varphi_t$. From the definition of the variables $c$ and $\theta$ we obtain the following parametrization of extremal trajectories.

If $\lambda \in C_1$, then
\begin{align}
x_1 (t) &= \frac{\sqrt{2} \sqrt{\alpha + E}}{\alpha} \Big(\sn \psi_t-\sn \psi_0\Big), \\
x_2 (t) &= -\frac{2 \ae}{\alpha} \Big(\E (\psi_t)- \E(\psi_0)\Big) - t, \\
y (t) &= \frac{\sqrt{\alpha + E}}{\sqrt{2} \alpha^2} \bigg(2 \ae \Big(\cn \psi_t \dn \psi_t-\cn \psi_0 \dn \psi_0 + \big(\E (\psi_t) - \E(\psi_0)\big) \big(\sn \psi_t + \sn \psi_0\big)\Big) \nonumber\\
&\qquad+  \alpha \Big(\sn \psi_t + \sn \psi_0\Big) t\bigg), \\
z(t) &=  \frac{\big(x_2(t)\big)^3}{6} + \frac{1}{3 \alpha^3}\bigg(2 \ae \Big(\big(\alpha + E\big) \big(\cn \psi_t \dn \psi_t (\sn \psi_t - \sn \psi_0) \nonumber\\
     &\qquad-2 (\cn \psi_t \dn \psi_t - \cn \psi_0 \dn \psi_0) \sn\psi_0\big) - \big(E + 3 (\alpha + E) \sn^2 \psi_0\big) \big(\E(\psi_t) \nonumber\\
		&\qquad- \E(\psi_0)\big)\Big) + \alpha \Big(\alpha - E - 3 \big(\alpha + E\big) \sn^2 \psi_0\Big) t\bigg).
\end{align}
Projections of extremals to the plane $(x_1, x_2)$ with $|\alpha|=1$ and $\psi_0 = 0$ are shown in Fig.~\ref{c123}, left. This case has no upper bounds for  time: $\tsupr (\lambda) = +\infty$.

\begin{figure}[htbp]
\centering
\includegraphics[width=0.175\linewidth]{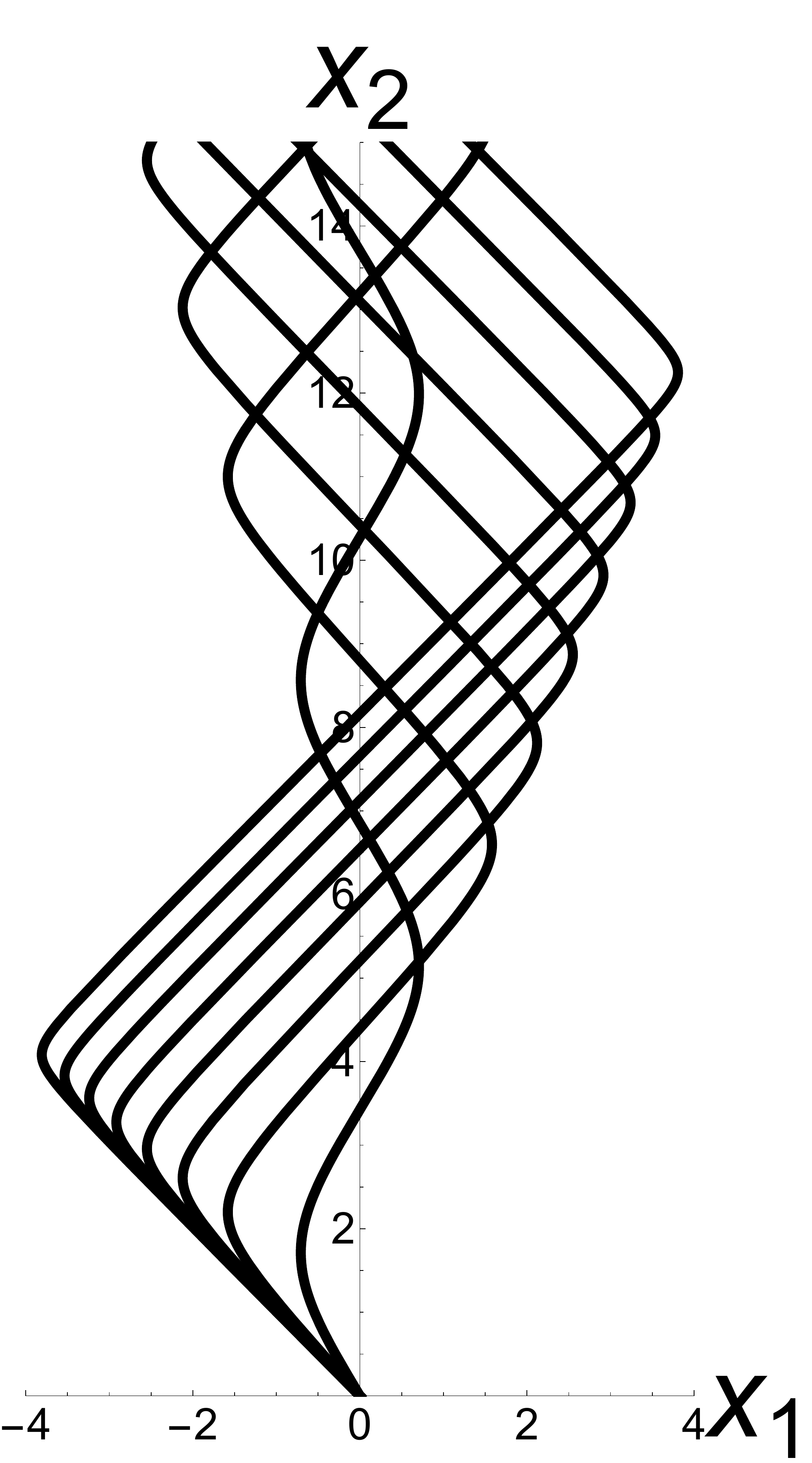}
\,
\includegraphics[width=0.39\linewidth]{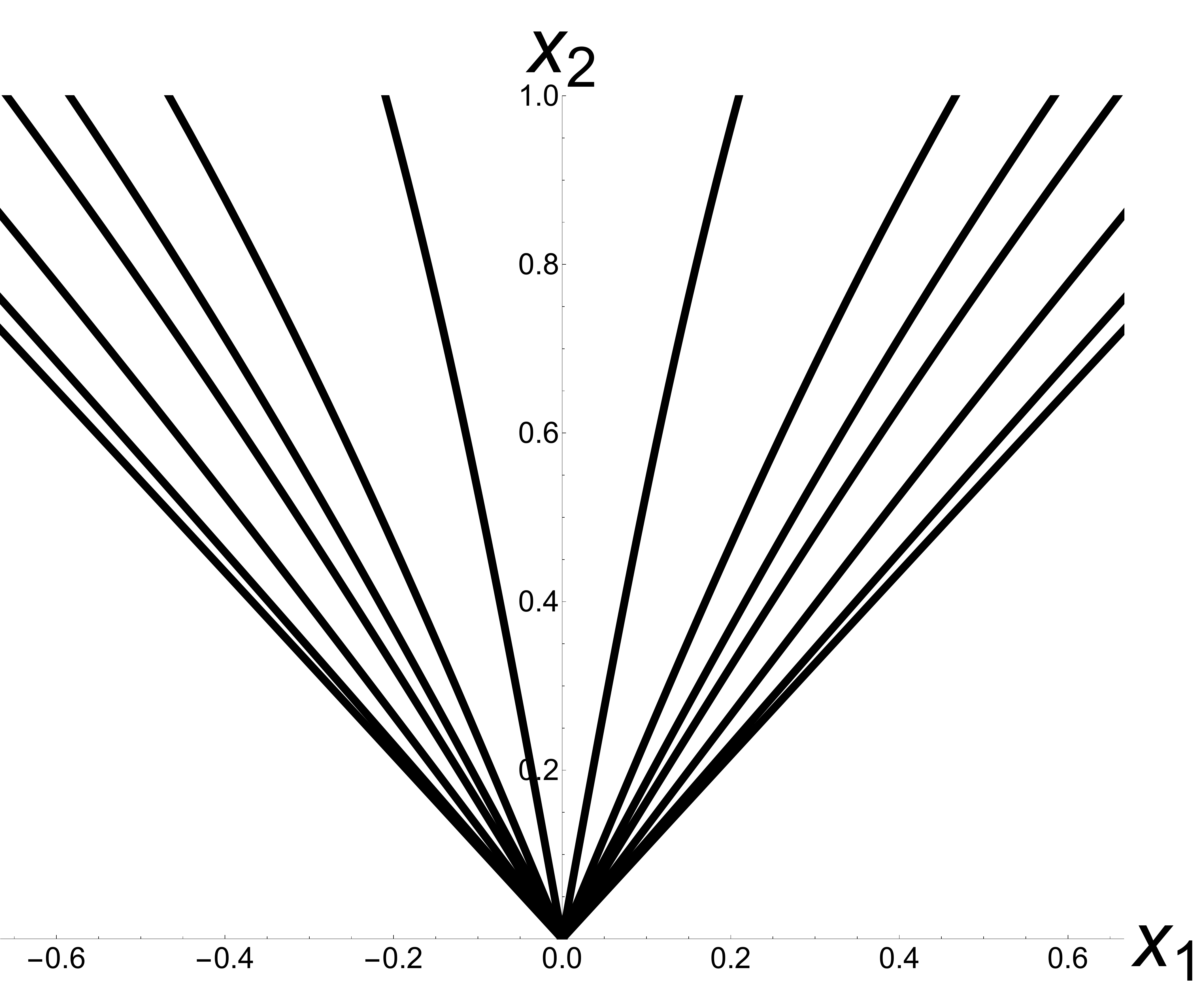}
\,
\includegraphics[width=0.39\linewidth]{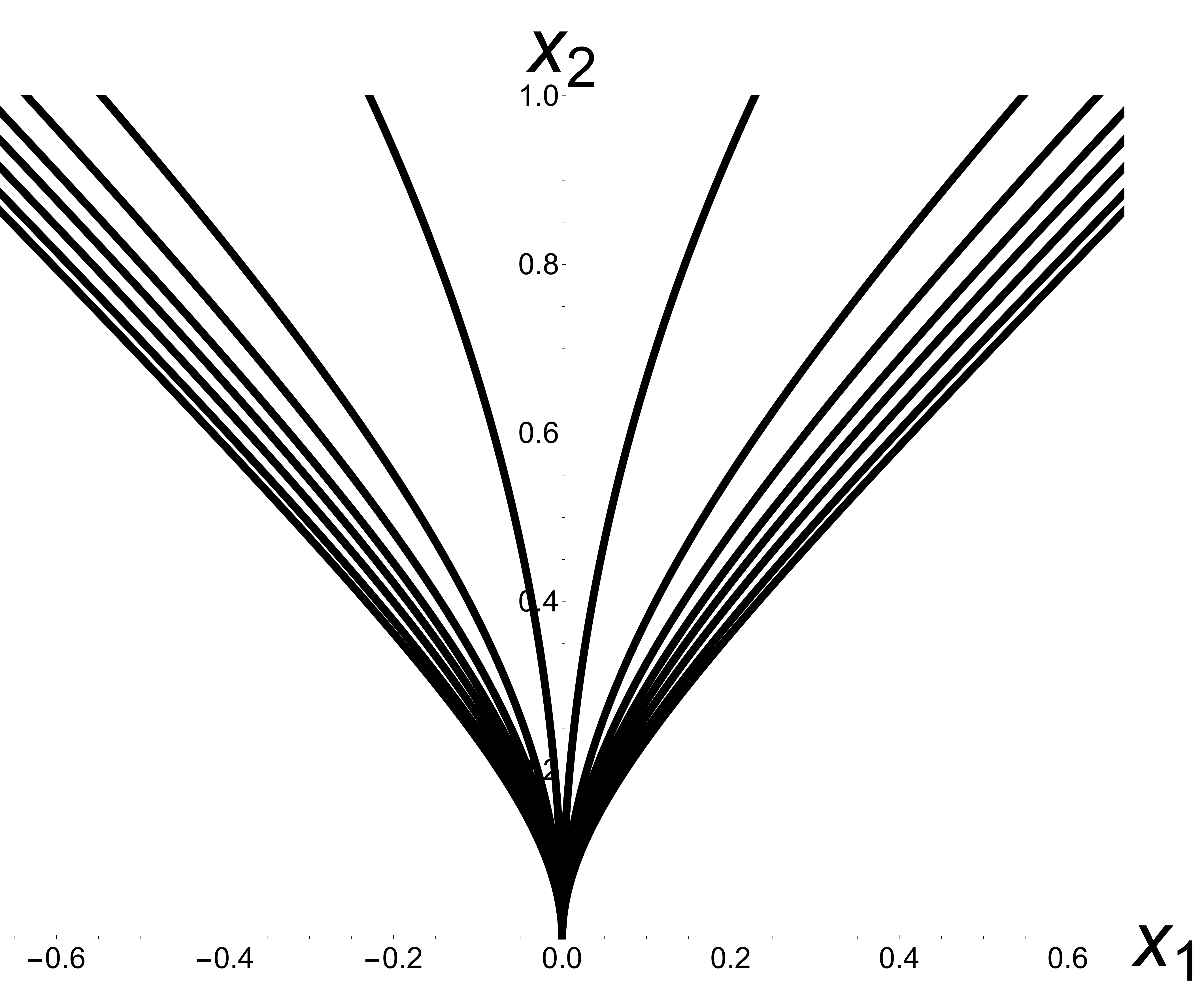}
\caption{Projections of goedesics to the plane $(x_1,x_2)$ for $\lambda \in C_1$, $\lambda \in C_2$ and $\lambda \in C_3$ (spacelike case)}\label{c123}
\end{figure}

If $\lambda \in C_2$, then
\begin{align}
x_1 (t) &= -\frac{\sqrt{2} \sqrt{-\alpha - E} \sgn \theta}{\alpha} \Big(\ssc \psi_t - \ssc \psi_0\Big) ,\\
x_2 (t) &= -\frac{2 \ae \big( \E(\psi_t) - \E (\psi_0) + \dn \psi_0 \ssc \psi_0 - \dn \psi_t \ssc \psi_t\big) + E t}{\alpha},\\
y (t) &= -\frac{\sqrt{-\alpha - E} \sgn \theta}{\sqrt{2} \alpha^2} \bigg(E \Big(\ssc \psi_0 + \ssc \psi_t\Big) t+2 \ae \Big( \big( \ssc \psi_0 + \ssc \psi_t\big) \big(\E (\psi_t) - \E(\psi_0)\big) \nonumber\\
  &\qquad+ \big(\dn \psi_t - \dn \psi_0\big) \big(1 - \ssc \psi_0 \ssc \psi_t\big)\Big)\bigg),\\
z (t) &= \frac{\big(x_2(t)\big)^3}{6}+\frac{2 \ae \big(3 (\alpha + E) \sn^2 \psi_0 - E \cn^2 \psi_0\big)}{3 \alpha^3 \cn^2\psi_0} \Big(\E(\psi_t) - \E(\psi_0)\Big) \nonumber\\
&\qquad+ \frac{1}{3 \alpha^3 \ae \cn^3 \psi_t \cn^3 \psi_0} 
\Bigg(2 \ae\bigg(\cn \psi_t \Big(\alpha \big( \cn^2 \psi_t \dn \psi_0 (1 - 3 \cn^2 \psi_0) \nonumber\\
&\qquad+ 3 \cn^2 \psi_0 \dn \psi_t\big) + \Big(\cn^2 \psi_t \dn \psi_0 \big(1 - 4 \cn^2 \psi_0\big) + 3 \cn^2 \psi_0 \dn \psi_t\Big) E\bigg) \sn \psi_0  \nonumber\\
&\qquad- \cn \psi_0 \dn \psi_t \bigg(\alpha \cn^2 \psi_0+ 3 \alpha \sn^2 \psi_0 \cn^2 \psi_t + \Big(3 \cn^2 \psi_t + \cn^2 \psi_0 \big(1 \nonumber\\
&\qquad- 4 \cn^2 \psi_t\big)\Big) E\bigg) \sn \psi_t\Bigg)
+\frac{\big(\alpha + E\big) \big(\cn^2 \psi_0 (\alpha - 4 E) + 3 E\big)}{3 \alpha^3 \cn^2 \psi_0}t.
\end{align}
Projections of extremals to the plane $(x_1, x_2)$ with $|\alpha|=1$ and $\psi_0 = 0$ are shown in Fig.~\ref{c123}, center.  Notice that since $\psi_t \in(-\K, \K)$ we have the upper bound for the time parameter:
\begin{align}
t < \tsupr (\lambda) = \frac{\K(k) - \psi_0}{\ae}. \label{supr2}
\end{align}

If $\lambda \in C_3$, then
\begin{align}
x_1 (t) &= \frac{2 \sgn c (\cn \psi_0 \dn \psi_t - \cn \psi_t \dn \psi_0)}{\ae \cn \psi_t \cn \psi_0 (1 - k_2)}, \\
x_2 (t) &= \frac{2}{\ae (1 - k_2)} \Big(\dn \psi_t \ssc \psi_t - \dn \psi_0 \ssc \psi_0 - \E(\psi_t)+\E(\psi_0)\Big) + t, \\
y (t)   &= \frac{\sgn c}{\ae (1 - k_2)} \bigg(\frac{2}{\ae (1 - k_2)} \Big(\big(\dc \psi_t + \dc \psi_0\big) \big(\E(\psi_t) - \E(\psi_0)\big) \nonumber\\
 &\qquad+ \big(\dc \psi_0 \dc \psi_t + k_2\big) \big(\sn \psi_0 - \sn \psi_t\big)\Big) -  \Big(\dc \psi_0 + \dc \psi_t\Big) t \bigg), \\
z(t)    &= \frac{2 \big(\E (\psi_t)-\E(\psi_0)\big) \big(6 - 6 k_2 + \cn^2 \psi_0 (1 + 7 k_2)\big)}{3 \ae^3 \cn^2 \psi_0 (-1 + k_2)^3}+ \frac{2 (3 \dn^2  \psi_0 + \cn^2 \psi_0 )}{3 \ae^2 \cn^2 \psi_0 (-1 + k_2)^2}t \nonumber\\
&\qquad+ \frac{2} {3 \ae^3 \cn^3 \psi_t \cn^3 \psi_0 (-1 + k_2)^3}\Big(\cn^3 \psi_t \dn \psi_0 \big(2 + \cn^2 \psi_0 - 2 k_2 + 7 \cn^2 \psi_0 k_2\big) \nonumber\\
&\qquad \times \sn \psi_0 - 6 \cn^2 \psi_0 \cn \psi_t \dn \psi_0 \big(k_2-1 \big) \sn \psi_t + 2 \cn^3 \psi_0 \dn \psi_t \big(k_2 -1 \big) \sn \psi_t \nonumber\\
&\qquad -\cn \psi_0 \cn^2 \psi_t \dn \psi_t \big(6 - 6 k_2 + \cn^2 \psi_0 (1 + 7 k_2)\big) \sn \psi_t\Big) +\frac{\big(x_2(t)\big)^3}{6}.
\end{align}
Projections of extremals to the plane $(x_1, x_2)$ with $|\alpha|=1$ and $\psi_0 = 0$ are shown in Fig.~\ref{c123}, right. Notice that since $\psi_t \in (-\K,\K)$ we have an upper bound for time parameter: $\tsupr (\lambda) = \frac{\K(k) - \psi_0}{\ae}$.

If $\lambda \in C_4$, then
\begin{align}
x_1(t) &= 4 \sgn \theta\bigg(\frac{e^{\psi_t}}{-1 + \alpha e^{2 \psi_t}}-\frac{e^{\psi_0}}{-1 + \alpha e^{2 \psi_0}}\bigg), \\
x_2(t) &= \frac{4}{\ae}\bigg(\frac{1}{-1 + \alpha e^{2 \psi}}- \frac{1}{-1 + \alpha e^{2 \psi}}\bigg)+t, \\
y(t)   &= - \sgn \theta \frac{2 e^{\psi_0} (-1 + \alpha e^{\psi_t+\psi_0})(2 + \ae t + e^{\psi_t-\psi_0}(\ae t -2))}{\ae (-1 + \alpha e^{2\psi_0} ) (-1 + \alpha e^{2 \psi_t})}, \\
z(t)   &= \frac{x_2^3(t)}{6} + \frac{4}{3}\Bigg( \frac{6 e^{2 \psi_0} t}{(-1 + \alpha e^{2 \psi_0})^2}+\frac{-1 - 9 \alpha e^{2 \psi_0} (-2 + \alpha e^{2 \psi_0})}{\alpha \ae (-1 + \alpha e^{2 \psi_0})^3}\nonumber\\
&\qquad+ \frac{12 \big(1 + \alpha e^{2 \psi_0} (-1 + 2 e^{\psi_t-\psi_0})\big)}{\alpha \ae (-1 + \alpha e^{2 \psi_0}) (-1 + \alpha e^{2 \psi_t})^2}
- \frac{8}{\alpha \ae (-1 + \alpha e^{2 \psi_t})^3} \nonumber\\
&\qquad- \frac{3 \Big(1 + \alpha e^{2 \psi_0} \big(6 + 4 e^{\psi_t-\psi_0} - \alpha e^{2 \psi_0} (-1 + 4 e^{\psi_t-\psi_0})\big)\Big)}{\alpha \ae  (-1 + \alpha e^{2 \psi_0})^2 (-1 + \alpha e^{2 \psi_t})}\Bigg).
\end{align}
Projections of extremals to the plane $(x_1, x_2)$ with $|\alpha|=1$ are shown on the left (for $\psi_0$<0) and in the center (for $\psi_0$<0) in Fig.~\ref{c46}.

Notice that here we have the condition $\ds \alpha e^{2 \sqrt{\alpha} \varphi_t} \neq 1$. Denote $\ds t_0 = -\frac{\ln \alpha}{2 \sqrt{\alpha} - \varphi_0}$. If $t_0 <0$, then there is no upper bound for time parameter: $\tsupr (\lambda) = + \infty$ (see Fig.~\ref{c46}, left). If $t_0>0$, then $\tsupr (\lambda) = t_0$  (see Fig.~\ref{c46}, center).

\begin{figure}[htbp]
\centering
\includegraphics[width=0.32\linewidth]{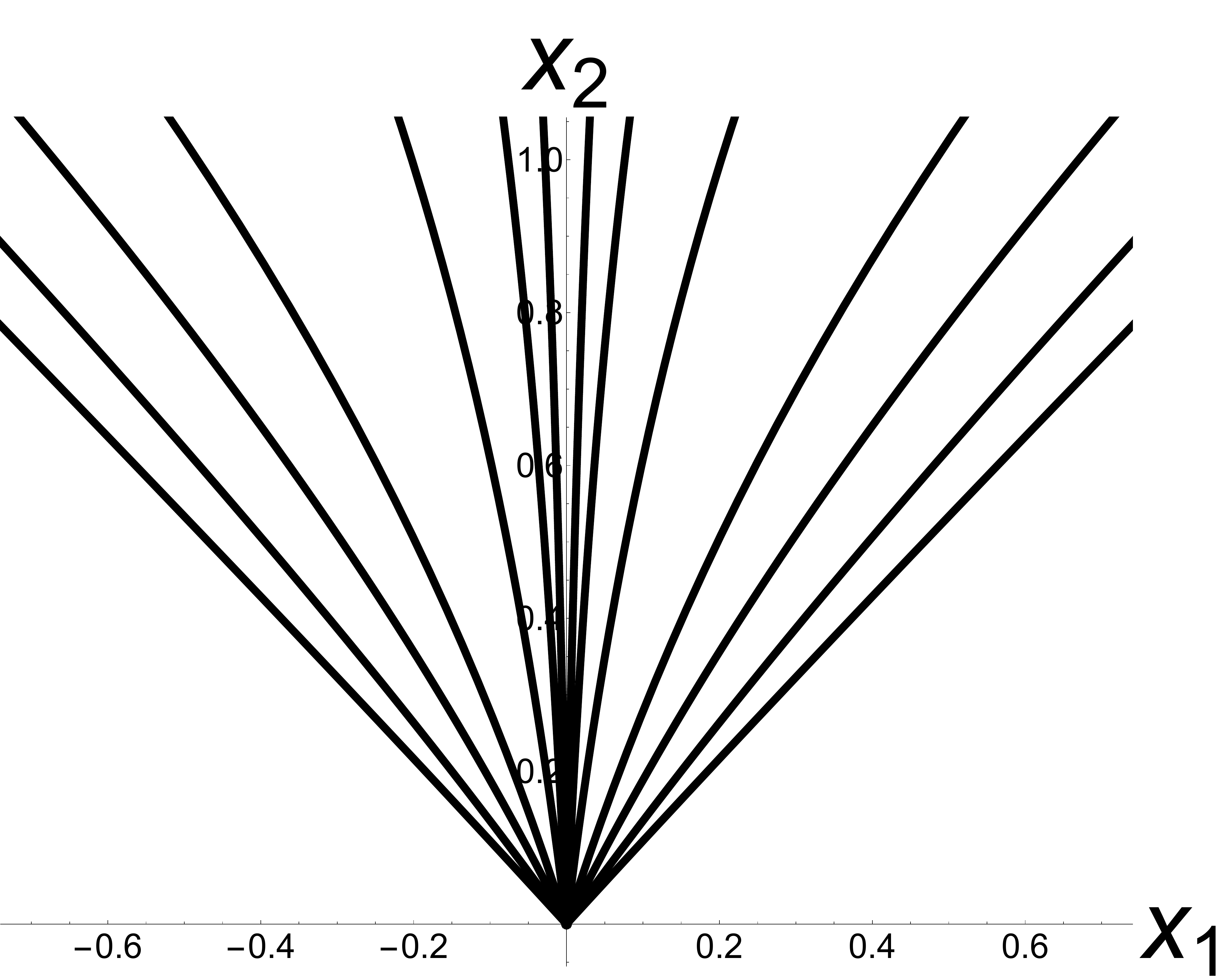}
\
\includegraphics[width=0.32\linewidth]{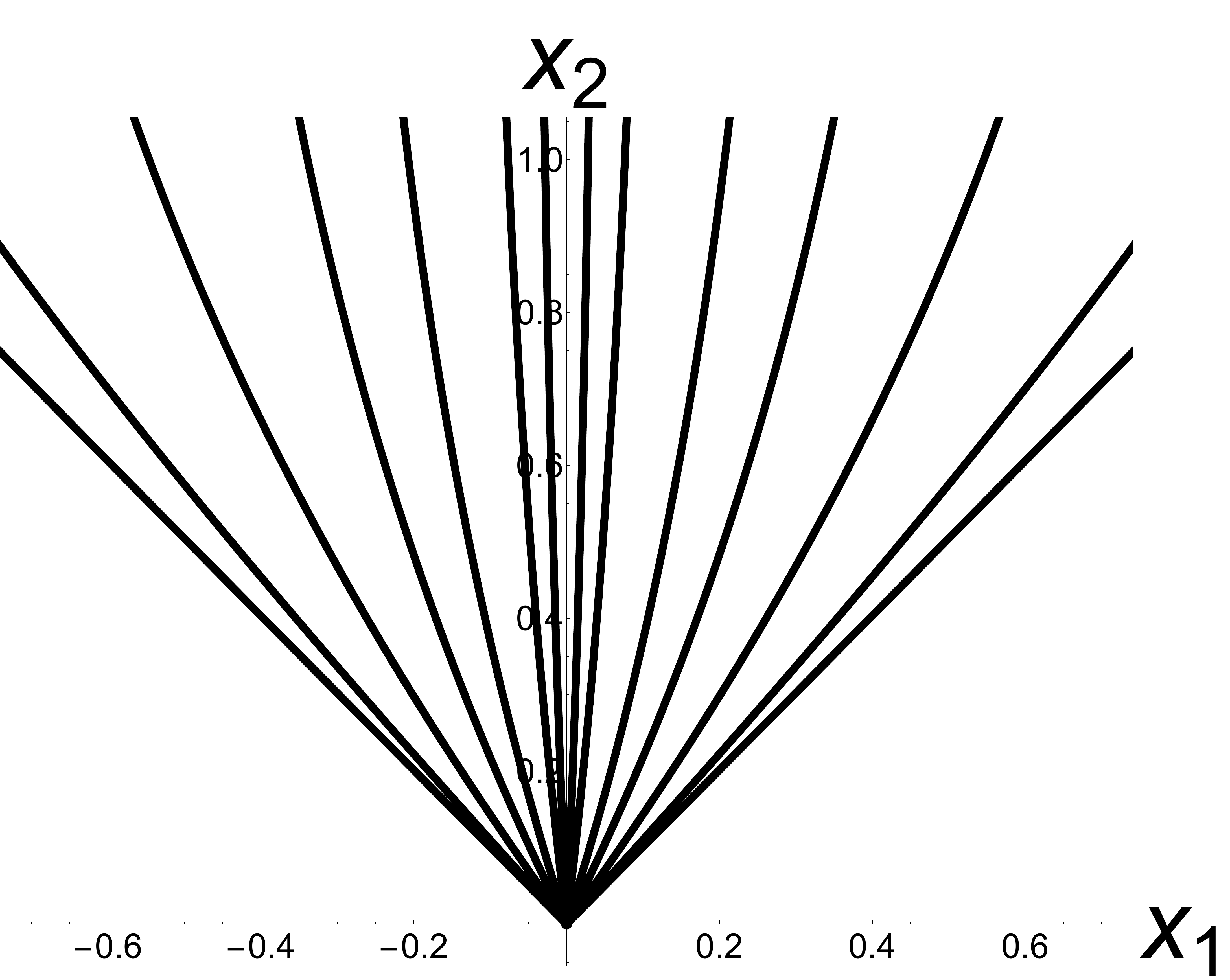}
\
\includegraphics[width=0.32\linewidth]{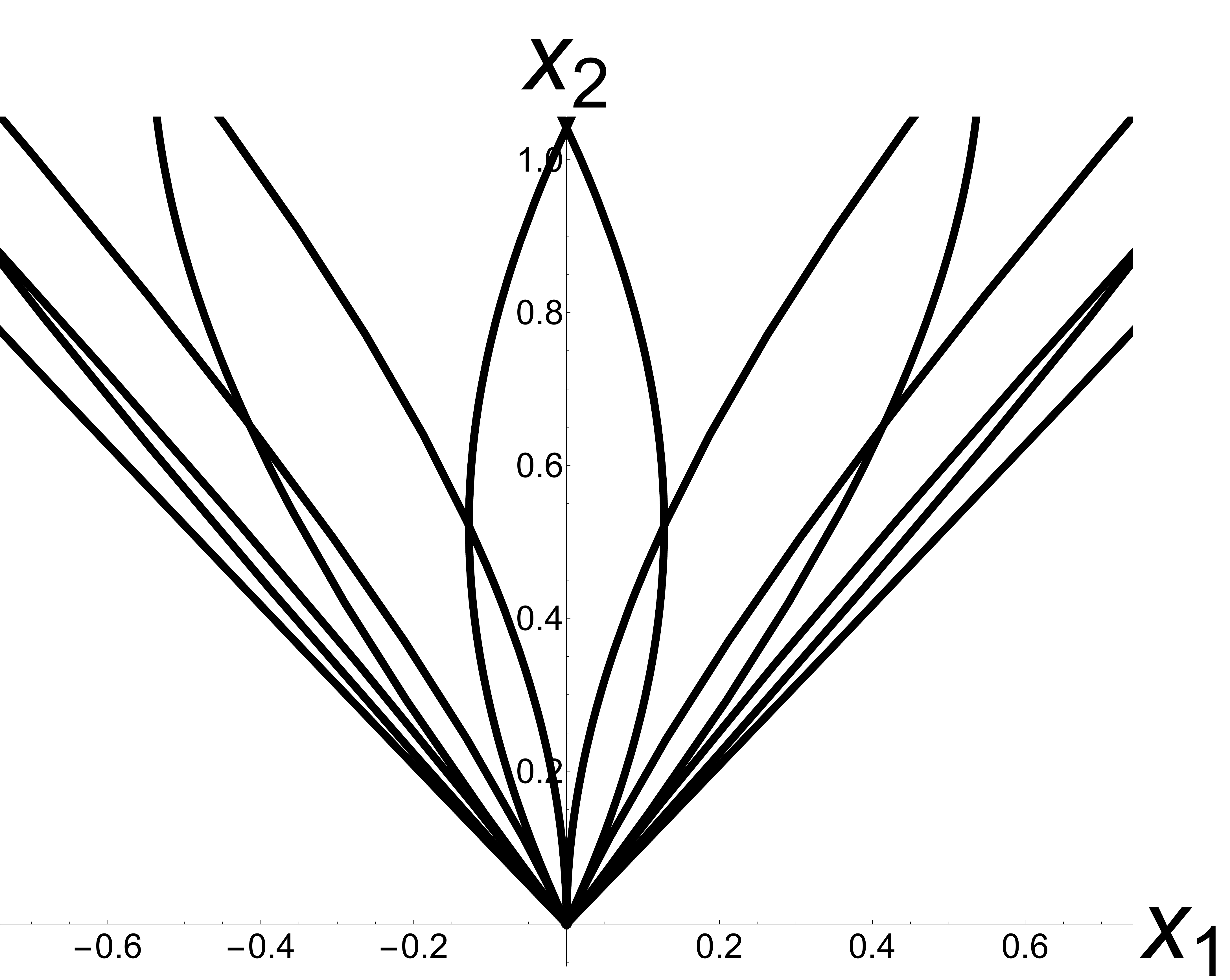}
\caption{Projections of goedesics to the plane $(x_1,x_2)$ for $\lambda \in C_4$ (2 types) and $\lambda \in C_6$ (spacelike case)
}\label{c46}
\end{figure}

In other cases the Hamiltonian system is easy to integrate. If $\lambda \in C_5$, then
\begin{align}
&x_1(t) \equiv 0,\quad x_2(t) = t, \quad y \equiv 0, \quad z = \frac{t^3}{6}.
\end{align}
Projection of the extremal to the plane $(x_1, x_2)$ is the vertical straight line. This case has no upper bounds for the time: $\tsupr = + \infty$.

If $\lambda \in C_6$, then
\begin{align}
&\alpha=0, \qquad  c \equiv \const \neq 0, \quad \theta = \theta_0 - c t, \nonumber\\
x_1(t) &= \frac{\cosh(c t- \theta_0) - \cosh \theta_0}{c}, \\
x_2(t) &= \frac{\sinh (c t - \theta_0) + \sinh \theta_0}{c}, \\
y(t) &= \frac{c t - \sinh (c t)}{2 c^2}, \\
z(t) &= \frac{4 \cosh\big(\frac{3 c t}{2} - 3 \theta_0\big) \sinh^3 \frac{c t}{2} - 3 \cosh \theta_0 \big(c t - \sinh(c t)\big)}{6 c^3}.
\end{align}
Projections of extremals to the plane $(x_1, x_2)$ with $|\alpha|=1$ and $\psi_0 = 0$ (hyperbolas) are shown in Fig.~\ref{c46}, right. This case has no upper bounds for   time: $\tsupr = +\infty$.

If $\lambda \in C_7$, then
\begin{align}
&\alpha=0, \qquad c \equiv 0, \quad \theta \equiv \const, \quad  \sinh \theta = s_0,\quad  \cosh \theta = c_0,\nonumber\\
x_1(t) & =  -s_0 t,  \\
x_2(t) & = c_0 t, \\
y(t)  &\equiv 0, \\
z(t) &=\frac{c_0(1+ 2 s_0^2)}{6} t
\end{align}
Projections of extremals to the plane $(x_1, x_2)$ are straight lines. This case also has no upper bounds for   time: $\tsupr = + \infty$.

So we obtained parameterization of the exponential mapping
\begin{align}
&\Exp: N \to M = \R^4,  \quad \Exp (\lambda, t)=q(t),\\
&N = \{(\lambda, t) \in C \times \R_+ ~ | ~ t \in \big(0,\tsupr(\lambda)\big)\},
\end{align}
in the spacelike case. It maps a pair $(\lambda, t)$ to the point of corresponding extremal $q(t)$.

Further we investigate discrete symmetries of the exponential mapping and on this basis find estimates for the cut time on extremals.

\subsection{Discrete symmetries of exponential mapping}
Subsystem~(\ref{dotth})-(\ref{dota}) for the costate variables of the normal Hamiltonian system has symmetries which preserve the direction field of this system.

\begin{figure}[htbp]
\centering
\includegraphics[width=0.46\linewidth]{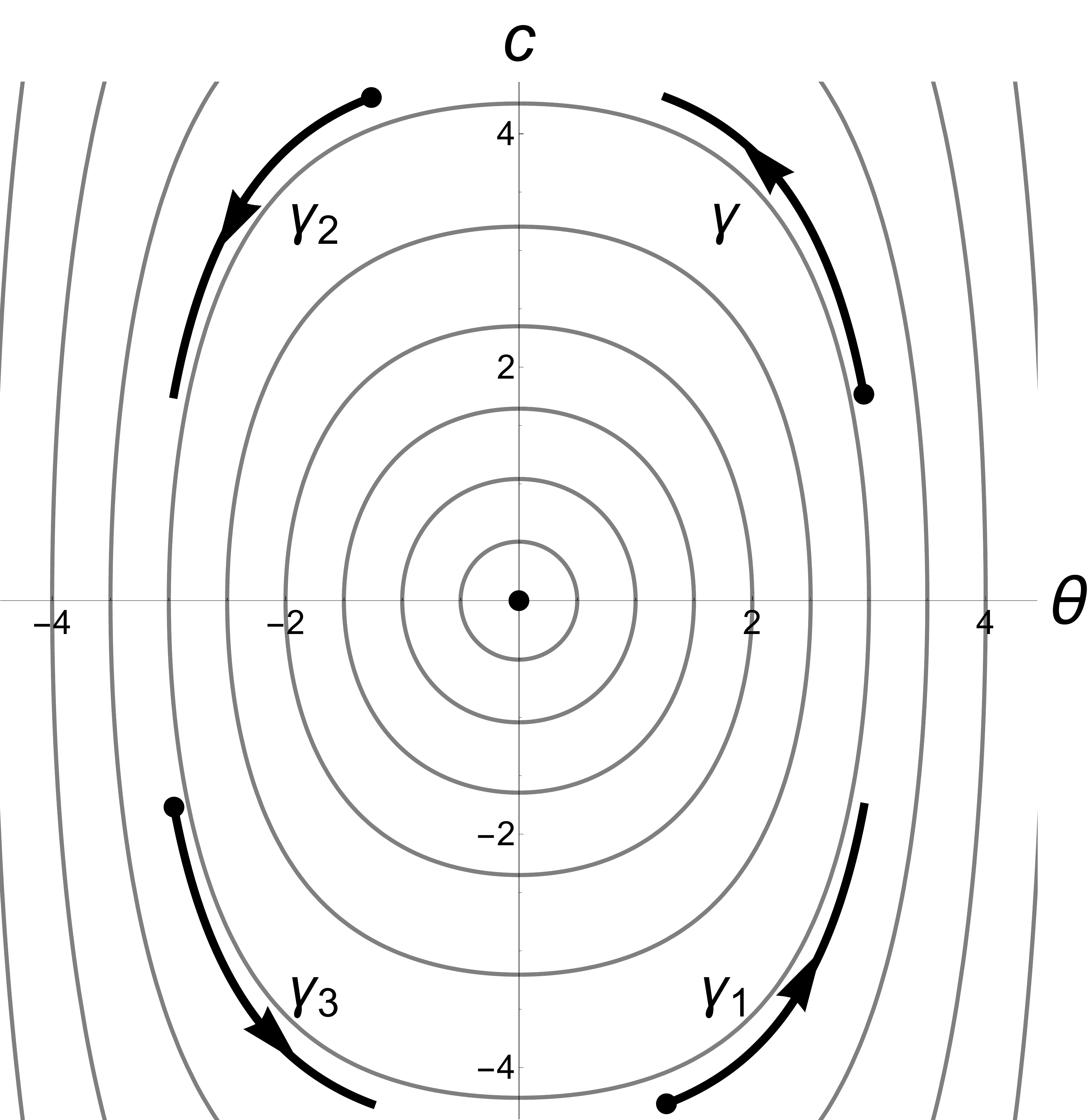}
\qquad
\includegraphics[width=0.46\linewidth]{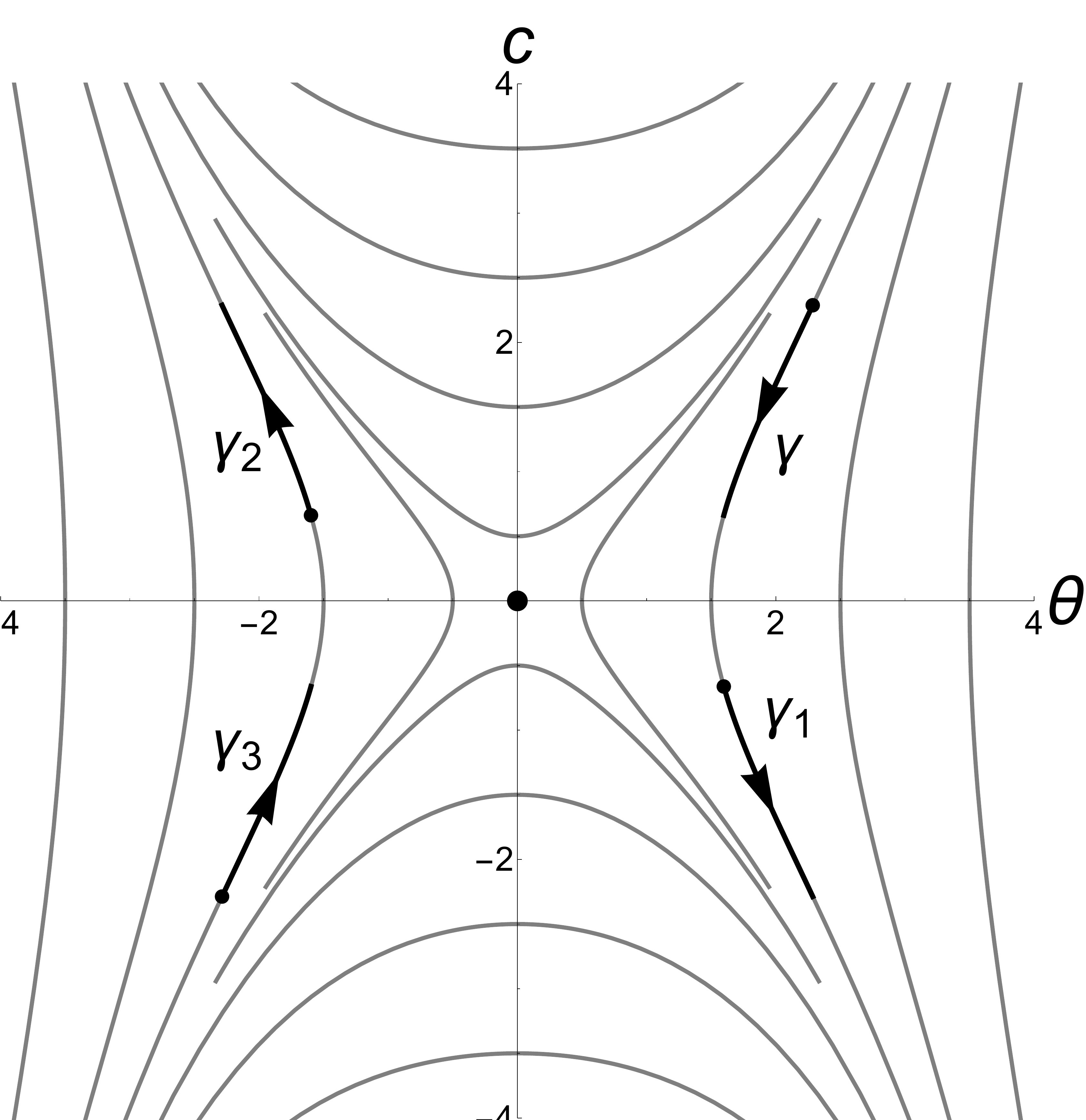}
\caption{Action of symmetries on trajectories of vertical subsystem for $\alpha = -1, 1$ (spacelike case)}\label{reflim}
\end{figure}

Let us describe the action of symmetries $\varepsilon^i$ on the set of trajectories of the vertical subsystem with preservation of time direction. Denote a smooth curve $$\gamma = \Big\{\big(\theta (t), c (t), \alpha \big) ~ | ~ t \in [0,T]\Big\} \subset C.$$  Define action of symmetries on these curves (see. Fig.~\ref{reflim}):
\begin{align*}
& \varepsilon^1 : \gamma \mapsto \gamma_1 = \Big\{\big(\theta^1(t), c^1(t), \alpha^1\big) ~ | ~ t \in [0, T]\Big\} = \Big\{\big(\theta (T-t), -c(T-t), \alpha\big)\Big\}, \\
& \varepsilon^2 : \gamma \mapsto \gamma_2 = \Big\{\big(\theta^2(t), c^2(t), \alpha^2\big) ~ | ~ t \in [0, T]\Big\} = \Big\{\big(-\theta(T-t), c(T-t), \alpha\big)\Big\}, \\
& \varepsilon^3 : \gamma \mapsto \gamma_3 = \Big\{\big(\theta^3(t), c^3(t), \alpha^3\big) ~ | ~ t \in [0, T]\Big\} = \Big\{\big(-\theta(t), -c (t), \alpha\big)\Big\}.
\end{align*}
It is obvious that if $\gamma$ is a solution to the vertical subsystem~(\ref{dotth})-(\ref{dota}), then $\gamma_i, i = 1, 2, 3,$ is a solution as well.

Consider the group of symmetries $G = \{ \mathrm{Id}, \varepsilon^1, \varepsilon^2, \varepsilon^3 = \varepsilon^1 \circ \varepsilon^2\} \cong \Z^2 \times \Z^2$. The symmetry $\varepsilon^3$ preserves direction of time, while $\varepsilon^1$ and $\varepsilon^2$ change it.

Continue the action of symmetries from the vertical subsystem to solutions of the Hamiltonian system of PMP
\begin{align*}
&\dot{\theta} (t) = - c (t), \quad \dot{c}(t) = - \alpha \sinh \theta (t), \quad \dot{\alpha} = 0, \\
&\dot{q} (t) = -\cosh \theta (t) X_1 \big(q(t)\big) + \sinh \theta (t) X_2 \big(q (t)\big),
\end{align*}
in the following way:
\begin{align}
\varepsilon^i: \Big\{\big(\theta (t), c (t), \alpha, q(t)\big) ~ | ~ t \in [0, T]\Big\} \mapsto \Big\{ \big(\theta^i(t), c^i(t), \alpha^i, q^i(t)\big) | t \in [0,T]\Big\}.\label{symact}
\end{align}

Let $q(t) = \big(x_1 (t), x_2 (t), y(t), z(t)\big), \ t \in [0,T],$ be a geodesic and
\begin{align*}
q^i(t) = \big( x_1^i(t), x_2^i (t), y^i (t), z^i(t)\big), \qquad t \in [0,T], \ i = 1, 2, 3,
\end{align*}
be its image under the action of a symmetry $\varepsilon^i$.

\begin{lemma}\label{lemx12} The symmetries $\varepsilon^i$ map trajectories in the plane $(x_1, x_2)$ in the following way:
\begin{align*}
&x_1^1(t) = x_1(T)-x_1(T-t), \qquad x_2^1(t) = x_2(T)-x_2(T-t), \\
&x_1^2(t) = x_1(T-t)-x_1(T), \qquad x_2^2(t) = x_2(T)-x_2(T-t), \\
&x_1^3(t) = -x_1(t), \qquad x_2^3(t) = x_2(t).
\end{align*}
\end{lemma}
\begin{proof}
Using direct integration, for instance, for $\varepsilon^1$ we verify that
\begin{align*}
&x_1^1(t) = \int_0^t \big(-\sinh \theta(T-s)\big) d s = -\int_{T-t}^{T} \big(-\sinh \theta(r)\big) d r = x_1(T) - x_1(T-t), \\
&x_2^1(t) = \int_0^t \cosh \theta (T-s) d s = -\int_{T-t}^{T} \cosh \theta(r) d r = x_2(T)  - x_2 (T-t).
\end{align*}
\end{proof}

\begin{lemma}\label{lemend} The symmetries $\varepsilon^i$ maps endpoints of geodesics $q = \big(x_1, x_2, y, z \big)$ to endpoints of $q^i = \big(x_1^i, x_2^i, y^i , z^i\big)$ in the following way:\
\begin{align*}
& x_1^1(T) = x_1(T),  &&x_2^1(T) = x_2(T), &&y^1(T) = -y(T), &&z^1(T) = z(T)-x_1(T) y(T), \\
& x_1^2(T) = -x_1(T), &&x_2^2(T) = x_2(T), &&y^2(T) = y(T), &&z^2(T) = z(T)-x_1(T) y(T), \\
& x_1^3(T) = -x_1(T), &&x_2^3(T) = x_2(T), &&y^3(T) = -y(T), &&z^3(T) = z(T).
\end{align*}
\end{lemma}
\begin{proof}
Lemma~\ref{lemx12} gives us expressions for $x_1^i(T)$ and $x_2^i(T)$. The expressions for other variables are obtained by integration. For example, for $\varepsilon^1$ we have
\begin{align*}
y^1 (T) &= \int_0^T \frac{x_2^1(t) \sinh \theta^1 (t) + x_1^1 (t) \cosh \theta^1 (t)}{2} d t=\frac{x_2(T)}{2}\int_0^T \sin \theta(T-t) d t \\
&\qquad+ \frac{x_1(T)}{2}\int_0^T \cosh \theta(T-t) d t - \int_0^T \frac{x_2(s) \sinh \theta (s)  + x_1(s) \cosh \theta (s)}{2} d s  \\
	&= -\frac{x_2(T) x_1(T)}{2} + \frac{x_1(T) x_2(T)}{2} - y(T)= -y(T), \\
z^1(T) &= \int_0^T \frac{\big(x_1^1(t)\big)^2 + \big(x_2^1(t)\big)^2}{2}\dot{x}_2^1(t) d t = \frac{\big(x_2^1(T)\big)^3}{6}+ \int_0^T \frac{\big(x_1^1(t)\big)^2}{2} \dot{x}_2^1 (t) d t  \\
& = \frac{\big(x_2(T)\big)^3}{2} + \int_0^T \frac{\big(x_1(T)\big)^2 - 2 x_1 (T) x_1(T-t) + \big(x_1 (T-t)\big)^2}{2} \dot{x}_2 (T-t) d t \\
&= \frac{\big(x_2 (T)\big)^3}{6} + \frac{\big(x_1 (T)\big)^2 x_2 (T)}{2} - x_1 (T) \int_0^T x_1 (s) \dot{x}_2 (s) d s + z(T) - \frac{\big(x_2 (T)\big)^3}{6} \\
& = \frac{\big(x_1 (T)\big)^2 x_2 (T)}{2} - x_1 (T) \Big(y(T) + \frac{x_1(T)x_2(T)}{2}\Big) + z(T) = z(T) -  x_1 (T) y(T).
\end{align*}
\end{proof}

We define the  action of $\varepsilon^i$ in the preimage of exponential mapping $N$ by restricting action to the initial point of trajectory of the vertical subsystem:
\begin{align}
&\varepsilon^i: C \to C, \qquad \varepsilon^i (\theta, c, \alpha) = (\theta^i, c^i, \alpha^i), \label{actsym3}\\
&(\theta^1, c^1, \alpha^1) = (\theta, -c , -\alpha), \\
&(\theta^2, c^2, \alpha^2) = (-\theta, c, \alpha), \\
&(\theta^3, c^3, \alpha^3) = (-\theta, -c, -\alpha), \label{actsym4}
\end{align}
in the following way:
\begin{align*}
&\varepsilon^i(\lambda, t) = \big(\varepsilon^i \circ e^{t \vec{H}_v} (\lambda), t\big),  \qquad i = 1, 2, \\
&\varepsilon^3(\lambda, t) = \big(\varepsilon^3 (\lambda), t\big),
\end{align*}
where $\vec{H}_v = -c \frac{\partial}{\partial \theta} - \alpha \sinh \theta \frac{\partial}{\partial c} \in \mathrm{Vec}(C)$ is the vertical part of the Hamiltonian vector field.

We define the action of $\varepsilon^i$ in the image $M$ of the exponential mapping $M$ by restricting the action to endpoints of geodesics (see Lemma~\ref{lemend}):
\begin{align}
&\varepsilon^i : M \to M,\quad \varepsilon^i (q) = \varepsilon^i (x_1, x_2, y, z) = q^i = (x_1^i, x_2^i, y^i, z^i), \label{actsym1}\\
&(x_1^1, x_2^1, y^1, z^1) = (x_1,  \ x_2, \ -y, \ z - x_1 y), \\
&(x_1^2, x_2^2, y^2, z^2) = (-x_1, \ x_2, \ y, \ z - x_1 y), \\
&(x_1^3, x_2^3, y^3, z^3) = (-x_1, \ x_2, \ -y, \ z).\label{actsym2}
\end{align}

Since the actions of $\varepsilon^i$ in the domain $N$ and the image $M$ of the exponential map are induced by the actions of symmetries~(\ref{symact}) on trajectories of Hamiltonian system, we have the following result.

\begin{proposition}\label{prop2}
The mappings $\varepsilon^i, i= 1, 2, 3,$ are symmetries of the exponential mapping, i.e.,
\begin{align*}
&(\varepsilon^i \circ \Exp) (\theta,c,\alpha,t) = (\Exp \circ \varepsilon^i) (\theta, c, \alpha, t), \qquad (\theta, c, \alpha, t) \in N.
\end{align*}
\end{proposition}

\subsection{Maxwell points of spacelike normal extremals}
In this section we compute the Maxwell points corresponding to the symmetries $\varepsilon^i$. On this basis we derive estimates for the cut time along extremals.

We define the Maxwell sets in the preimage of $\Exp$ corresponding to the symmetries $\varepsilon^i$:
\begin{align*}
\MAX^i = \big\{(\lambda,t) \in N ~ | ~ \lambda^i \neq \lambda, \ \Exp(\lambda^i, t) = \Exp(\lambda, t) \big\}, \quad \lambda = (\theta, c, \alpha), \ \lambda^i = \varepsilon^i(\lambda).
\end{align*}

It follows from Proposition~\ref{prop2} that the equality $\Exp (\lambda^i, t) = \Exp (\lambda, t)$ is equivalent to $\varepsilon^i\big(q(t)\big) = q(t)$. Therefore we get a description of fixed points of the symmetries $\varepsilon^i$ in the preimage of the exponential mapping.

\begin{lemma}~

\begin{enumerate}
\item $\varepsilon^1(q) = q ~ \Longleftrightarrow ~ y = 0$.
\item $\varepsilon^2(q) = q ~ \Longleftrightarrow ~ x_1 = 0$.
\item $\varepsilon^3(q) = q ~ \Longleftrightarrow ~ x_1^2 + y^2 = 0$.
\end{enumerate}
\end{lemma}
\begin{proof}
Follows from definition~(\ref{actsym1})--(\ref{actsym2}) of action of the symmetries in $M$.
\end{proof}

Now we compute fixed points of symmetries in the preimage of the exponential mapping, which is necessary for description of Maxwell sets. Introduce the following coordinates in the set $\cup_{i=1}^3 N_i, \ N_i = \{(\lambda, t) \in C_i \times \R_+ ~ | ~ t \in \big(0, \tsupr (\lambda)\big)\}$:
\begin{align}
&\tau = \frac{\psi_t + \psi_0}{2}, \qquad p = \frac{\psi_t-\psi_0}{2} = \frac{\ae t}{2}. \label{ptau2}
\end{align}
The parameter $\tau$ corresponds to the middle point of a trajectory.

\begin{lemma}~

\begin{enumerate}
\item $\lambda^1 = \lambda  ~ \Longleftrightarrow ~ \left\{\begin{array}{l}
			\sn \tau = 0, \text{ if } \lambda \in C_1 \cup C_2, \\
			\text{impossible, if } \lambda \in C_3 \cup C_4 \cup C_6,  \\
			c=0, \text{ if } \lambda \in C_5 \cup C_7. \end{array} \right.$
\item $\lambda^2 = \lambda ~ \Longleftrightarrow ~  \left\{\begin{array}{l}
			\cn \tau = 0, \text{ if } \lambda \in C_1, \\
			\text{impossible, if } \lambda \in C_2 \cup C_4, \\
			\sn \tau = 0, \text{ if } \lambda \in C_3, \\
			\theta = 0, \text{ if  } \lambda \in \bigcup_{i=5}^7 C_i.\end{array} \right.$
\item $\lambda^3 = \lambda ~ \Longleftrightarrow ~  \left\{\begin{array}{l}
			\text{impossible, if } \lambda \in (\bigcup_{i=1}^4 C_i) \cup C_6, \\
			\theta = c = 0 \text{, if } \lambda \in C_5 \cup C_7.\end{array} \right.$
\end{enumerate}
\end{lemma}
\begin{proof}
Follows from definition of action of the symmetries in the domain of the exponential mapping~(\ref{actsym3})--(\ref{actsym4}) and definition~(\ref{ptau2}) of the coordinate $\tau$.
\end{proof}

The equations $\varepsilon^i (q) = q$ define sub-manifolds in $\R^4$ with dimension from 3 to 2, the exponential mapping transforms the corresponding Maxwell sets into these sub-manifolds:
\begin{align*}
\Exp (\MAX^i) \subset \{ q \in \R^4 ~ | ~ \varepsilon^i (q) = q\}\quad i=1,2,3.
\end{align*}

Further we describe in detail the sets $\MAX^1, \MAX^2$ since they define sub-manifolds in $M$ with the maximum dimension 3, while $\MAX^3$ defines a sub-manifold with dimension 2. We investigate roots of the equations $x_1=0, y=0$ along a geodesic. Using   coordinates~(\ref{ptau2}) for the case $\lambda \in C_1$ we get
\begin{align*}
x_1 &= \frac{2 \sqrt{2 (\alpha + E)} \cn \tau \dn \tau \sn p}{\alpha (1 - k^2 \sn^2 p \sn^2 \tau)}, \\
y &= \frac{2 \sqrt{2 (\alpha + E)} \sn \tau f_1(p)}{\alpha^2 \ae (1 - k^2 \sn^2 p \sn^2 \tau)},\\
f_1 (p) &= \alpha p \cn p \dn p + \ae^2 \big(2 \cn p \dn p \E (p) - (1 + k^2) \sn p + 2 k^2 \sn^3 p\big).
\end{align*}

This gives us description of the Maxwell sets in $N$:
\begin{align}
\MAX^1 \cap N_1 &= \{\nu \in N_1 \mid y = 0, \sn \tau \neq 0\} =\{\nu  \in N_1  \mid f_1(p)=0\},\\
\MAX^2 \cap N_1 &= \{\nu \in N_1 \mid x_1 = 0, \cn \tau \neq 0\} = \{\nu \in N_1  \mid \sn p=0\},
\end{align}
where $\nu = (\tau, p, E, \alpha)$.

Let $p_1^i, i \in \N$, denote a positive root of equation $f_1(p)=0$, s.t. $p_1^1< p_1^2 < \dots < p_1^i<\cdots $. We have $t_{\MAX^1} (\lambda) = p_1^1$, $t_{\MAX^2} (\lambda)= 2 \K$. There is a conjecture (this conjecture is supported numerically) that $t_{\MAX^2} (\lambda) < t_{\MAX^1} (\lambda)$. This means that each geodesic from $C_1$ meets the set $\MAX^2$ first.

If $\lambda \in C_2$, then
\begin{align}
x_1 &= -\frac{2 \sgn \theta \sqrt{2} \sqrt{-\alpha - E} \dn \tau \cn p \sn p}{\alpha (\cn^2 \tau - \sn^2 p \dn^2 \tau)}, \nonumber\\
y &= \frac{2 \sgn \theta \sqrt{2} \sqrt{-\alpha - E} \cn \tau \sn \tau f_2(p) }{\alpha^2 \ae (\cn^2 \tau - \sn^2 p \dn^2 \tau)}, \nonumber\\
f_2(p) &= -E (p) \dn p - \ae^2 \big(2 \dn p \E (p) - k^2 \cn p \sn p\big) \label{fy2}
\end{align}

We have
\begin{align}
\MAX^1 \cap N_2  &= \{\nu \in N_2 \mid y = 0, \sn \tau \neq 0, \tau \in (-\K,\K), p \in (0,\K)\} \nonumber\\
&\subset \{\nu  \in N_2 \mid f_2(p)=0,  p \in (0,\K) \}, \label{max12}\\
\MAX^2 \cap N_2 &= \{\nu \in N_2 \mid x_1 = 0, \tau \in (-\K,\K), p \in (0,\K)\} = \emptyset,
\end{align}
where $\nu = (\tau, p, E, \alpha)$.

\begin{lemma} \label{lemc2}
The function $f_2(p)>0$ for $p \in \big(0,\K(k)\big), k \in (0,1)$.
\end{lemma}
\begin{proof}
We show that $g(p) = \dn p$ is a comparison function for $f_2(p)$ for $p \in (0,\K)$.

The inequality $f_2(p)\not \equiv 0$ follows from the expansion $f_2(p) = \frac{1}{3}\alpha k^2 p^3 + o(p^3)$.
Notice that $g(p) > 0$ for $p \in (0, \K)$. Finally we get the equalities
\begin{align*}
& \bigg(\frac{f_2(p)}{g(p)}\bigg)' = \frac{\alpha^2 \sn^2 p \cn^2 p} {\ae^2 \dn^2 p}, \qquad \frac{f_2(p)}{g(p)} = \frac{1}{3}\alpha k^2 p^3 + o(p^3).
\end{align*}

So $g(p)$ is a comparison function for $f_2 (p)$; thus, it follows from Lemma~\ref{lem4} that $f_2 (p) > 0$ for $p \in  (0, \K)$.
\end{proof}

\begin{proposition} \label{propc2}
$\MAX^1 \cap N_2 = \emptyset$.
\end{proposition}
\begin{proof}
Follows immediately from Lemma~\ref{lemc2} and (\ref{max12}).
\end{proof}

If $\lambda \in C_3$, then
\begin{align}
x_1 &= \frac{4 \sgn c \sn p \sn \tau}{\ae (\cn^2 \tau - \dn^2 \tau \sn^2 p)}, \nonumber\\
y &= -\frac{4 \sgn c \cn \tau \dn \tau f_3(p)}{\ae^2 (1 - k_2)^2 (\cn^2 \tau - \dn^2 \tau \sn^2 p)}, \nonumber\\
f_3(p) &= \cn p \dn p \big((1 - k_2) p - 2 \E (p) \big) + (1 + k_2) \sn p - 2 k_2 \sn^3 p. \label{fy3}
\end{align}

We have
\begin{align}
\MAX^1 \cap N_3  &= \{\nu \in N_3 \mid y = 0, \tau \in (-\K,\K), p \in (0,\K)\} \nonumber\\
&=\{\nu  \in N_3 \mid f_3(p)=0, p \in (0,\K) \}, \label{max13}\\
\MAX^2 \cap N_3  &= \{\nu \in N_3 \mid x_1 = 0, \sn \tau \neq 0, \tau \in (-\K,\K), p \in (0,\K)\} = \emptyset\nonumber,
\end{align}
where $\nu = (\tau, p, E, \alpha)$.

\begin{lemma} \label{lemc3}
The function $f_3(p)>0$ for $p \in (0,\K)$.
\end{lemma}
\begin{proof}
We show that $g(p) = \cn p \dn p$ is a comparison function for $f_3(p)$ for $p \in (0,\K)$.

The inequality $f_3(p)\not \equiv 0$ follows from the expansion $\ds f_3(p) = \frac{\alpha^2}{3 \ae^4} p^3+ o(p^3)$.
Notice that $g(p) > 0$ for $p \in (0, \K)$. Finally we get the equalities
\begin{align*}
& \bigg(\frac{f_3(p)}{g(p)}\bigg)' = \frac{(1 - k_2)^2 \sn^2 p}{\cn^2 p \dn^2 p}, \qquad \frac{f_2(p)}{g(p)} = \frac{\alpha^2}{3 \ae^4} p^3+ o(p^3).
\end{align*}

So $g(p)$ is a comparison function for $f_3 (p)$; thus, it follows from Lemma~\ref{lem4} that $f_3 (p) > 0$ for $p \in  (0, \K)$.
\end{proof}
\begin{remark}
Notice that if $k^2 = k_2 < 0$, then Lemma $\ref{lemc3}$ is valid. We have the following formulas:
\begin{align*}
x_1& = \frac{2 \sqrt{2} \sqrt{\alpha + E} \sgn c \dn \tau \sn \tau \dn p \sn p}{\alpha (\cn^2 \tau - \dn^2 \tau \sn^2 p)}, \\
y &= \frac{2 \sqrt{2} \cn \tau \sqrt{\alpha + E} \sgn c f_4 (p)}{\alpha \ae (\cn^2 \tau - \dn^2 \tau \sn^2 p)},  \qquad \ae=\hat{\ae}=\sqrt{\alpha}, \\
f_4 (p)&= 2 \cn p \E (p) - p \cn p - \dn p \sn p, \qquad k = \hat{k} = \sqrt{\frac{\alpha - E}{2 \alpha}}.
\end{align*}
One can easily prove that $g(p) = \cn p$ is a comparison function for $f_4(p)$ for $p \in (0,\K)$.
\end{remark}
\begin{proposition}
$\MAX^1 \cap N_3 = \emptyset$.
\end{proposition}
\begin{proof}
Follows immediately from Lemma~\ref{lemc3} and (\ref{max13}).
\end{proof}
We found the following upper bound for the cut time along general extremal with $\lambda \in \bigcup_{i=1}^3 C_i$:
\begin{theorem}
\begin{align}
&\lambda \in C_1 ~ &&\Longrightarrow ~ &&\tcut (\lambda) \leq  2 \K(k). \\
&\lambda \in C_2 \cup C_3 ~ &&\Longrightarrow ~ &&\tcut (\lambda) \leq  \frac{\K(k) - \psi_0}{\ae}.
\end{align}
\end{theorem}
There is a conjecture that this bound is sharp.

\section{Conclusion}
In this paper we consider extremals in the Engel group with a sub-Lorentzian metric. For this problem

$\bullet$ we prove that timelike normal extremals are locally maximizing;

$\bullet$ we calculate the timelike, spacelike and lightlike normal extremal trajectories;

$\bullet$ we  describe the symmetries of the exponential map and the corresponding
Maxwell points of timelike, spacelike normal extremals;

$\bullet$  we  obtain an upper estimate for the cut time on extremal trajectories.

In the future we are intending to investigate the optimality of extremal trajectories, relying on the method of nonlinear
approximation we shall also apply our results to consider general
non-linear systems with 2-dimensional control in a 4-dimensional space with a sub-Lorentzian metric.

\end{document}